\documentclass{article}
\usepackage{graphicx}
\usepackage{amssymb, amsthm}
\usepackage{amsmath}
\usepackage{enumitem}
\usepackage{comment}
\usepackage[backend=biber,style=alphabetic,sorting=ynt]{biblatex}
\usepackage{subfig}

\newtheorem{theo}{Theorem}[section]
\newtheorem{theorem}[theo]{Theorem}
\newtheorem{lemma}[theo]{Lemma}
\newtheorem{prop}[theo]{Proposition}
\newtheorem{corollary}[theo]{Corollary}
\newtheorem{fact}[theo]{Fact}

\newtheorem{claim}[theo]{Claim}

\newtheorem{obs}{Observation}

\newtheorem*{Inv}{Invariants}

\newtheorem{dfn}[theo]{Definition}

\newtheorem{rem}[theo]{Remark}
\newcommand{\N}{\ensuremath{\mathbb{N}}} 
\newcommand{\Z}{\ensuremath{\mathbb{Z}}}

\newcommand{\model}{\ensuremath{\mathcal{M}}}
\newcommand{\nodel}{\ensuremath{\mathcal{N}}}
\newcommand{\lang}{\ensuremath{\mathcal{L}}}
\newcommand{\set}[1]{\{#1 \}}
\newcommand{\tp}{\textnormal{tp}}
\newcommand{\qftp}{\textnormal{qftp}}

\newcommand{\res}{\textnormal{res}}

\newcommand{\U}{\ensuremath{\mathcal{U}}}

\newcommand{\clos}{\rangle_{\{+,-,\frac{\cdot}{n}\}}}

\newenvironment{claimproof}[1]{\par\noindent\emph{Proof:}\space#1}{\hfill $\square_{Claim}$ \bigskip}

\let\strokeL\L
\DeclareRobustCommand{\L}{\ifmmode\mathbf{L}\else\strokeL\fi}

\title{Limit Semigroups with 2 Generators}

\author{Estrada, Felipe \and Onshuus, Alf \and Rincón, David}

\begin{document}

\nocite{*}
\maketitle

The object of this paper is to study numerical semigroups, by which we mean cofinite submonoids of $(\N, +, <, 0)$. In particular, we will study numerical semigroups with two generators. Each of these is quite well understood, with known and simple formulas for the Frobenius number and the genus. The initial idea of this project was proposed by Tristram Bogart at a joint Combinatorics-Model Theory seminar at Universidad de los Andes: he suggested that if one looked at the limit theory of randomly chosen semigroups, it was possible that some of the ``noise'' present in small finite cases would be removed and one would have an idea of a common theory for randomly selected numerical semigroups. 

This paper is about the study of limit semigroups with 2 generators. As we will show in this paper, there is no unique limit theory, since instead of ``removing the noise'', in order to understand the limit theory of semigroups one needs to find proofs that work uniformly in all (or most) of the semigroups one is studying. So our study points out invariants that govern the structure of numerical semigroups generated by two elements. As such, it may be useful in order to pinpoint the different cases one would need to consider when studying general numerical semigroups. 

\medskip

There are various ways to suggest what ``random'' would mean. First, notice that given two elements $a, b\in \mathbb N$, the monoid generated by $(a/d, b/d)$, where $d=(a,b)$ is the greatest common divisor of $a$ and $b$, is a numerical semigroup.  Since there is no way to endow $\mathbb N$ with a uniform non zero measure, what one can do is for each $N\in \mathbb N$ take random (with the uniform measure) elements $a_N,b_N\leq N$, consider the semigroup generated by $\frac{a_N}{(a_N, b_N)}$ and $ \frac{b_N}{(a_N, b_N)}$ and understand what the ``limit semigroup'' would look like as $N$ goes to infinity.

The idea of ``limit semigroup'' (or limit theory) is where first order logic (and model theory) comes into play, as it has in various applications in limit combinatorics: Ultralimits. The advantage of ultralimits is that once we fix an ultraproduct $\mathcal U$ over $\N$, ultralimits allow us to define, given any sequence $S_i$ of ordered monoids, an ordered monoid which will be the limit of the sequence $S_i$.

As ordered monoids, any semigroup is naturally a structure with a binary function $+$ and a binary relation $<$. An \emph{ultrafilter} over $\N$ is a set $\mathcal U$ of subsets of $\N$ closed under supersets and intersections, with the property that for any subset $X\subseteq \N$ we have $X\in\mathcal U$ if and only if $\N\setminus X\not\in \mathcal U$. Another way of understanding $\mathcal U$ is that $\mathbb N$ admits $0-1$ additive measures; an ultrafilter is just the sets of measure $1$. 

Once an ultrafilter $\mathcal U$ is fixed, for any given monoids $S_i$, one can define a monoid $S_{\mathcal U}:=\Pi_{i\in N} S_i/\mathcal U$ with the property (\L o\'s' Theorem) that any first-order sentence $\phi$ will be true in $S_{\mathcal U}$ if and only if
\[
\{i\mid \phi \text{ is true in }S_i\}\in \mathcal U.
\]
As we shall see, in the first section, if every $S_i$ is a semigroup generated by two elements, then $S_{\mathcal U}$ behaves very much as what one would expect an ``infinite semigroup with two generators'' to behave. Furthermore, if we choose  $S_i$ to be a random semigroup which results after choosing random elements $a,b$ in the interval $[2,i]$ then  $S_{\mathcal U}$ will no longer be a numerical semigroup (it will not be a subset of $\N$ nor will it be cofinite) but it will have all the properties shared by the majority (in term of $\U$) of the semigroups $S_i$. Studying these objects (which obviously depend on the semigroups $S_i$ and on the ultrafilter $\mathcal U$) will shed light on what the structure of a random semigroup with two generators might look like. We will call any such object a \emph{limit semigroup}. This paper studies the properties of limit semigroups, with questions such as what is the structure of limit semigroups, what are the first order properties shared by all numerical semigroups which are inherited by the limit semigroups, whether the theory is definable, and what the axioms and invariants might be.

\bigskip

From now on, we will refer to semigroups generated by two elements as 2-semigroups. We will abuse notation and sometimes use this term to refer to both the combinatorial (standard) semigroups generated by two coprime elements $a,b\in \mathbb N$, and to the ultralimits of these. When confusion might arise, we will refer to the ones generated by finite elements as standard semigroups, and to the ultralimits as limit 2-semigroups. 

\bigskip

The paper is divided as follows: In Section \ref{Section 1} we will recall some of the known properties of standard 2-semigroups and prove that they can be stated by first order sentences, which implies that they will hold in any limit 2-semigroup. We will then give an axiom sketch of the common (first order) theory of all 2-semigroups (both standard and by \L o\'s' Theorem of all limit 2-semigroups as well). 

This common theory is not complete: To start with, given a prime number $p\in \N$, both the congruence modulo $p$ of the generators $a$ and $b$ will be coded by first order theories, which already gives continuum many possibilities for the possible theories of limit 2-semigroups. In Section \ref{Section invariants} we find invariants which we know are needed to classify all possible theories of all limit 2-semigroups. 

We then start studying the theory of the limit semigroups, and consider the first steps towards the main model theoretic tool we will use to prove completeness in some cases: quantifier elimination. We will do this in Section \ref{Section: model theoretic}. We will also point out why, even though (standard) 2-semigroups are well understood and in general the first order theory of numerical semigroups is complete (the first order theory of any finite structure is always complete), understanding the first order theory of any limit 2-semigroup entails model theoretic questions that are still not known.

We will then understand when these invariants are enough to get a full classification of the theories. It will depend on whether or not once we fix these invariants and the axioms they entail, we get a complete theory (one where all first order sentences true in the limit 2-semigroups are proved by the axioms). In Section \ref{Section quantifier elimination} we give a necessary condition for this to hold. We will then exhibit examples of classes of theories of limit 2-semigroups where the resulting theory is indeed complete. 

Finally, in Section \ref{Section: conclusions} we will exhibit conditions under which the limit theory of certain numerical 2-semigroups exists and is well understood.

\section{(Standard) Numerical Semigroups}\label{Section 1}

In this section, we exhibit the common theory of standard and (by \L o\'s' Theorem) limit 2-semigroups.

Let $\lang = \set{+, <, 0}$ be the language of ordered monoids. From now on, we will only work in the first order theory with the fixed language $\mathcal L$.

\subsection{Known results in standard 2-semigroups and the corresponding first order sentences}

\begin{dfn}
    Let $S$ be a standard numerical semigroup. The smallest nonzero element of $S$ is called the \emph{multiplicity} of the semigroup.  The greatest natural number not in $S$ is called the \emph{Frobenius number} of $S$. The least number $c$ in $S$ such that every natural number greater than or equal to $c$ belongs to $S$ is called the \emph{conductor} of $S$.
    
    A set of natural numbers is called a \emph{set of generators} for the semigroup if every element of the semigroup can be written as a finite $\N$-linear combination of generators.
\end{dfn}

\begin{fact}
    Every standard numerical semigroup has a minimal finite set of generators, which is unique. In a standard semigroup with 2 generators, the conductor $c$ of the semigroup relates to the generators $a,b$ by $c=ab-a-b+1$ (\cite{Gar1}).
\end{fact}

The coprimality of the generators on a 2-semigroup allows us to uniformly define some \emph{notable elements and predicates} using only the language $\lang$, so that they will also be definable in any ultraproduct of 2-semigroups.

\begin{prop}\label{basic facts}
Let $S$ be a standard ordered 2-semigroup with generators $0<a<b$.
    \begin{enumerate}

        \item The generators $a,b$ of $S$, the conductor $c$ of $S$, and $ab$ are uniformly definable constants. The conductor $c$ is the least element such that whenever $x, y$ are such that $x+c \leq y$, then there is a $z$ such that $x+z = y$.
        
        \item There are definable predicates $M_a(x), M_b(x)$ which hold if and only if $x=ab$, or $x < ab$ and $x$ is a multiple of $a$ or $b$, respectively. $M_a(x)\wedge M_b(x)$ implies $x=0\vee x=ab$.

        \item (Unique Decomposition) Every element $s$ of $S$ can be uniquely written as $s=m(ab)+ m_aa + m_bb$ with $m, m_a, m_b\in \mathbb N$ and $m_a <b, m_b < a$.
        
        On the set $[0, ab)$ the functions $\pi_a, \pi_b$ mapping $s\in S$ to its respective $m_aa, m_bb$ are uniformly definable.

        \item The set of all $x$ such that $\exists y(x = ny)$ is definable and is the domain of the ``division by $n$'' function. 

        \item The set of all pairs $(x,y)$ such that $\exists z (y + z = x)$ is definable and is the domain of subtraction. 
        
        \item The functions $\alpha$ and $\beta$ that map an element $s$ in $[0, ab]$ to the greatest element less than or equal to $s$ which satisfies $M_a$, and to the smallest element greater than or equal to $s$ which satisfies $M_b$, respectively, are definable. In particular, $x-\alpha(x)$ is the residue of $x$ modulo $a$, and $\alpha(ab)=\beta(ab)=ab$. 

        \item For all natural numbers $n,r$, with $r<n$, there exists a definable predicate $R_{n,r}$ such that for any $s\in S$, $s\equiv r~(\text{mod n})$ if and only if $R_{n,r}(s)$ holds.

        \item There are definable constants $\alpha_1,\beta_1$ such that $\alpha_1$ is the least (and unique) element in $S$ such that $M_a(\alpha_1),M_b(\beta_1)$ holds, and $\beta_1=\alpha_1+1$. In particular, the residue of $\beta_1$ modulo $a$ is 1.

        \item For any two elements $x,y$, and any natural number $r$ such that $r|(x-y)$ and $\frac{x-y}{r}<a$, there exist $\beta$ a multiple of $b$ such that the residue of $\beta$ modulo $a$ is $\frac{x-y}{r}$. In particular, $\frac{x-y}{r}=\beta-\alpha(\beta)$, where $\alpha(\beta)$ is the greatest multiple of $a$ less than $\beta$.
    \end{enumerate}
\end{prop}
\begin{proof}
Some of these are immediate. We outline why the others hold.
    \begin{enumerate}
        \item In any semigroup $S$, the smallest generator $a$ of a semigroup is its least nonzero element. The next generator, $b$, is the least element which is not a multiple of $a$, that is to say the least nonzero element that cannot be written as $x + a$ for some $x$. These properties can be expressed as first-order formulas.

        For defining the conductor $c$, note that if two elements $x, y$ are such that $x+c \leq y$, then there must be a $z$ such that $x+z = y$ (i.e., the element $y-x$ belongs to the semigroup because $y-x \geq c$). Further, the formula $$\phi(c') := \forall x,y ((y \geq c' + x) \Rightarrow \exists z (y = z + x ))$$
        is satisfied if and only if $c' \geq c$. so that $c$ may be defined as the least element which satisfies $\phi$.

        The constant $ab$ is the least element that can be written both as $y+a$ and $z+b$ for some elements $y, z$ in $S$.

        \item Let $M_a(x):= x = ab \lor \forall y \ \neg (y+b=x)$ and let $M_b(x):= x = ab \lor \forall y \ \neg (y+a=x)$.

        \item By definition any $s \in S$ may be written as $k_a a + k_b b$ for some $k_a, k_b \in \N$. Write $k_a$ as $n_a b + m_a$ with $m_a$ the residue of $k_a$ modulo $b$ (so $m_a<b$) and similarly let $n_b$ and $m_b$ be such that $k_b = n_b a + m_b$ with $m_b < a$. Now, $s = (n_a + n_b)ab + m_aa + m_bb$ which proves the existence of the decomposition.

        For uniqueness, suppose there is a second representation $s = m'ab + m_a'a + m_b'b$. Then $m_aa = m_a'a$ (mod $b$) and since $a,b$ are coprime, $m_a = m_a'$ (mod $b$). Both $m_a$ and $m_a'$ are smaller than $b$ so they must be equal. By an analogous argument $m_b = m_b'$, which implies $m = m'$.

        \item [7.] For any element $s\in S$, we have $s\equiv r~(\text{mod $n$})$ if and only if $nc+s\equiv r~(\text{mod $n$})$. Such congruence holds if and only if there is $t\in S$ such that $nt=nc+s+(n-r)$ because any natural number greater than or equal to $c$ is in $S$. For the same reason, the element $nc+s+(n-r)$ is the $(n-r)$-th successor of the element $nc+s$ in $S$ so the existence of $t$ can be expressed as a first-order formula depending on $r$ and $n$.

        \item [8.] The set $\{0,b,2b,\dots,(a-1)b\}$ is a complete system of residues modulo $a$, so one and only one of them must be congruent to $1$ modulo $a$. Say $\beta_1=nb\equiv 1~(\text{mod $a$})$ with $n<a$. The element $nb$ is the minimal element of $S$ in its residue class modulo $a$, and since $a|nb-1$, we have that $\alpha_1=nb-1\in S$. Hence, $\alpha_1$ is the least multiple of $a$ such that its successor $\beta_1$ is a multiple of $b$ and the distance between any elements $x,y\in S$ is greater than or equal to the distance between $\alpha_1$ and $\beta_1$. All this can be expressed in first order. 

        \item [9.] This is immediate from 8 by considering the elements $\frac{x-y}{r}\alpha_1$ and $\frac{x-y}{r}\beta_1$.
    \end{enumerate}
\end{proof}

\begin{rem}\label{alphas and betas functions}
 For any $x\in M_b$ we have   $\beta(\alpha(x))=x$. It is not true, however, that $\alpha(\beta(x))=x$ for any $x\in M_a$.
\end{rem}

\begin{rem}\label{alphas and betas}
    Analogous to (8) of the above proposition, for every natural number $k<b$ there are minimal elements $\alpha_k,\beta_k$ in $S$, multiples of $a$ and $b$, respectively, such that $\beta_k-k=\alpha_k$. Moreover, since $k\cdot \alpha_1 \equiv k~(\text{mod $b$})$, it follows that $\alpha_k=k\cdot \alpha_1-nab$ and $\beta_k=k\cdot \beta_1-nab$, for some $n\in\N$ with $n<k$.
\end{rem}

\begin{dfn}
    Let $S$ be a 2-semigroup. By definition (Proposition \ref{basic facts}(2)), $M_a(S)$, and $M_b(S)$ are the sets of all multiples of $a$ less than or equal to $ab$, and all multiples of $b$ less than or equal to $ab$, respectively (equivalently, all realizations of the formulas $M_a(x)$ and $M_b(x)$ in $S$). 
    
    We extend this definition and define $M_a^*(S)$ to be the set (not definable in first order) of all elements of $S$ with unique decomposition whose $b$-component is $0$, and $M_b^*(S)$ analogously.
\end{dfn}


\begin{rem}\label{alphafunction}
    It follows from the definition of $\alpha$ (Proposition \ref{basic facts}(6)) that for any $x$ with unique decomposition $x=nab+x_a+x_b$, the greatest element $y$ of $M_a^*(\model)$ such that $y\leq x$ is $y=nab+x_a+\alpha(x_b)$. Analogously for the $\beta$ function. It's worth noting that in general it is not true that $\alpha(x_1+x_2)=\alpha(x_1)+\alpha(x_2)$, but we do have $\alpha(x_1+x_2)=\alpha(x_1)+\alpha(x_2) +\epsilon a$, where $\epsilon$ is either $0$ or $1$, depending on how close $x_1$ and $x_2$ are to the previous multiple of $a$. A similar result applies for the $\beta$ function as well.
\end{rem}





\subsection{First Order Common Theory of 2-semigroups}

We will refer to the theory $Pr=Th(\N,+,<,0,1)$ as Presburger arithmetic. $Pr$ is the theory of ordered semigroups with one generator, and so it is to be expected that 2-semigroups share some of its structure with $Pr$ models. We recall the axiomatization of Presburger arithmetic.

\begin{dfn}
    \fussy Let $\lang=\{+,<,0,1,(D_n)_{n\in\omega}\}$ be the language consisting of a binary operation $+$, a binary relation $<$, constants $0,1$, and unary predicates $(D_n)_{n\in\omega}$. Presburger arithmetic is the theory defined by the following axioms:
    \begin{enumerate}
        \item Ordered abelian semigroup axioms.
        \item $0<1$
        \item $\forall x[x=0\lor x\geq 1]$
        \item $\forall x,y \left[ y\geq x\Rightarrow \exists z (y=z+x)\right]$. This $z$ is unique and we shall refer to it as $y-x$.
        \item $$\left\{\forall x \left[D_n(x)\iff \exists y \left(x=\underbrace{y+\dots+y}_{\textrm{$n$ times}}\right)\right]\right\}_{n\in \mathbb N}$$
        \item $$\left\{\forall x\left[\bigvee_{i=0}^{n-1}\left(D_n(x+\underbrace{1+\dots+1}_{\textrm{$i$ times}})\land\bigwedge_{j\neq i}\neg D_n(x+\underbrace{1+\dots+1}_{\textrm{$j$ times}})\right)\right]\right\}_{n\in \mathbb N}$$
    \end{enumerate}  
\end{dfn}

\begin{fact}\label{Presburger}
    Presburger arithmetic is a complete, decidable theory with quantifier elimination. See \cite{Mar}.
\end{fact}

\begin{dfn}
Consider a non standard model $\mathcal Z$ of Presburger arithmetic, and let $d\in \mathcal Z$ be an infinite element (meaning larger than any finite addition of 1's). Consider the substructure $\mathcal Z_d$ of $\mathcal Z$ with universe $[0,d]:=\{x\in \mathcal Z \mid 0\leq x\leq d\}$, together with the order, the predicates $D_n$ and the restricted addition, so that addition is only defined on the set $\{(x,y)\mid x\leq d-y\}$. We define the resulting theory as \emph{bounded Presburger}.

Its axioms are the same as those of Presburger except for the restriction of addition to the set described above, the fact that $d$ is the maximum element, and that $d$ is larger than any finite addition of 1's. 

We also need to replace Axiom $6_n$ of Presburger with two different axioms:

$$\left\{\forall x\left[\bigvee_{i=0}^{n-1}\left(D_n(x+\underbrace{1+\dots+1}_{\textrm{$i$ times}})\land\bigwedge_{j\neq i}\neg D_n(x+\underbrace{1+\dots+1}_{\textrm{$j$ times}})\right)\right]\right\}_{n\in \mathbb N}$$ for all $x$ such that  $x+\underbrace{1+\dots+1}_{\textrm{$n$ times}})\leq d$ and 
$$\left\{\forall x\left[\bigvee_{i=0}^{n-1}\left(D_n(x\underbrace{-1-\dots-1}_{\textrm{$i$ times}})\land\bigwedge_{j\neq i}\neg D_n(x\underbrace{-1-\dots-1}_{\textrm{$j$ times}})\right)\right]\right\}_{n\in \mathbb N}$$ for all $x$ such that  $x\underbrace{-1-\dots-1}_{\textrm{$n$ times}})\geq 0$.
\end{dfn}

The following is a slight variation of Proposition 4.8 in \cite{Dario}:

\begin{fact}\label{Garcia}
   Let $\mathcal Z$ be a model of bounded Presburger. Then for any $A\subset \mathcal Z_a$ the type $\tp^{\mathcal Z_a}(A)$ is implied by $\qftp(A,a)$. 
\end{fact}

\begin{proof}
 In \cite{Dario} this was proved for the structure $([0,a), +_a, <)$ where $+_a$ is addition modulo $a$. Clearly both structures are quantifier free equivalent. Alternatively, one can repeat the proof in \cite{Dario} reducing the problem to elimination of quantifiers in $\mathcal Z$. 
\end{proof}

\begin{rem}\label{presburger}
    In a standard 2-semigroup, the structures $(M_b(S),+,<,0,b)$ and $([0,a],+,<,0,1)$ are isomorphic (via multiplication by $b$). Hence, we can extend the language to include divisibility predicates in $M_b(S)$, so that the resulting structure is a segment of a model of $Pr$. Such divisibility predicates allow one to consider natural residue predicates $R_{n,r}^b$ in $M_b(S)$. An analogous analysis works for $(M_a(S),+,<,0,a)$.
        
    Formally, for $x\in M_a(S)$, $R_{n,r}^a(x)$ holds if and only if $x=na'+ra$ for some $a'\in M_a(S)$, and similarly,  
    $R_{n,r}^b(x)$ holds for $x\in M_b(S)$ if and only if $x=nb'+rb$ for some $b'\in M_b(S)$.
    
   (Notice also that the structures $(M_a^*(S),+,<,0,a,(R_{n,0}^a)_{n\in\omega})$ and \\ $(M_b^*(S),+,<,0,b, (R_{n,0}^b)_{n\in\omega})$ are models of Presburger arithmetic.)
\end{rem}

In light of the remark above, we shall make the following definition.

\begin{dfn}
    Let $\model$ be a 2-semigroup generated by $0<a<b$. For $x\in M_a^*(\model)$, we define the $M_a$-residues of $x$ to be $\tp^\model(x)\vert_{\{R_{n,r}^a\}}$. For $x\in M_b^*(\model)$, we define the $M_b$-residues of $x$ to be $\tp^\model(x)\vert_{\{R_{n,r}^b\}}$. For $x\in\model$, we define the $R$-residue type of $x$ to be $\tp^\model(x)\vert_{\{R_{n,r}\}}$.
\end{dfn}

\medskip

When it helps clean up the notation, we will work with the residue modulo $n$ as a function and use $\res_n(x)=r$ instead of $R_{n,r}(x)$. Similarly for $\res_n^a(x)$ and $\res_n^b(x)$.

\medskip

We will refer to all such predicates satisfied by an element as the residue type of the element, and allow the distinction between the $M_a$-residue, $M_b$-residue or $M$-residue, referring to the predicates $R_{n,r}^a, R_{n,r}^b$, and $R_{n,r}$, respectively. For an element $x\in M_a^*(\model)$, we will refer as its $M_a$-residue modulo $n$ to the number $i$ such that $R_{n,i}^a(x)$ holds in $\model$. Similarly for the $M_b$-residue modulo $n$ of an element $x\in M_b^*(\model)$.

\begin{dfn}
    For every $n\in\omega$, we define the partial function $\pi_a^n$ on every element $x\in[nab,(n+1)ab)$ as the function that retrieves the $a$-component of $x$. That is, if $x=m(ab)+x_a+x_b$ is the unique decomposition of $x$, $\pi_a^n(x):=x_a$. Similarly, we define $\pi_b^n(x):=x_b$. Such definitions are clearly expressible in first order.
\end{dfn}

For convenience, we will extend the language $\lang$ to include all previous predicates and functions:

\begin{dfn}
    The augmented language $\hat{\lang}\supseteq \lang$ denotes the language of ordered monoids with constant symbols $\{a,b,\alpha_1, \beta_1, ab\}$, unary function symbols $\alpha, \beta, \pi_a, \pi_b$, unary predicates $(R_{n,r})_{n,r\in \omega}, (R_{n,r}^a)_{n,r\in \omega}$ and $(R_{n,r}^b)_{n,r\in \omega}$. We refer to the predicates $R_{n,r}$, $R_{n,r}^a$, $R_{n,r}^b$ as \emph{residue predicates}. 
    
    It will also be convenient to add symbols for division by $n$ and a difference operator restricted to $R_{n,0}^a(M_a)$ and $R_{n,0}^b(M_b)$. So we will add unary functions $(\frac{\cdot}{n})_{n \in \omega}$ and a binary function $-$ whose domains are $R_{n,0}^a(M_a)\cup R_{n,0}^b(M_b)$ and 
    $\{(x,y)\in M_a \mid x<y\}\cup \{(x,y)\in M_b \mid x<y\}$, respectively.
 \end{dfn}

In the following statements, as in the rest of the paper, $n(x)$ and $r(x)$ are just abbreviations for adding $x$ to itself $n$ and $r$ times, respectively.

\begin{dfn}\label{dfn axioms}
    Let $\hat{\lang}$ be the extended language as defined above. For convenience, for $i\in\Z$, let $s^i(x)$ denote the $i$-th successor (or predecessor if $i$ is negative) of $x$ (axioms will be included for the existence of such elements). The $\hat{\lang}$-theory $T_{ons}$ of 2-semigroups with generators $0<a<b$ is axiomatized by the following axioms:
\begin{enumerate}
    \item Axioms of discretely ordered abelian semigroup.
    \item Definition of $a,b,\alpha_1,\beta_1$, and $ab$ constants (as per Proposition \ref{basic facts}).
    \item Properties of the conductor $c$: $ab+1=c+a+b$ and any two numbers whos distance is greater than $c$ are subtractable (where adding $1$ just denotes applying to successor function once).
    \item Existence of $a$-predecessor or $b$-predecessor for any positive element:
    $$\forall x>0 \exists y[y+a=x \lor y+b=x].$$

    \item Axioms of bounded Presburger arithmetic on the sets $M_a$, $M_b$ of multiples of $a$ and $b$,  replacing $D_n$ for $R_{n,0}^a$ and $R_{n,0}^b$, respectively. Additionally, the behavior of the auxiliary symbols $\frac{\cdot}{n}, R_{n,r}^a$ and $R_{n,r}^b$. This is:

 For every $n,i,j\in \N$, we have   \begin{itemize} 
    \item The domain of  $\frac{\cdot}{n}$ is $R_{n,0}^a(M_a)\cup R_{n,0}^b(M_b)$.  
    \item $n(\frac{x}{n})=x$ for any $x$ in $R_{n,0}^a(M_a)\cup R_{n,0}^b(M_b)$.
  \item $\forall x [R_{n,r}^a(x)\iff (M_a(x)\land \exists y(M_a(y)\land x=n(y)+r(a)))].$
    \item $\forall x [R_{n,r}^b(x)\iff (M_b(x)\land \exists y(M_b(y)\land x=n(y)+r(b)))].$
\end{itemize}
    
    \item Definition of $-$, including its domain.
    \item Any interval of the form $[x, y)$, where $y \geq x+a$, contains a multiple of $a$ and any interval of the form $[x, y)$, where $y \geq x+b$, contains a multiple of $b$:
    $$\forall x \exists y [x+a<ab\Rightarrow (M_a(y)\land x\leq y< x+a]$$
    $$\forall x \exists y [x+b<ab\Rightarrow (M_b(y)\land x\leq y< x+b]$$
    
    \item Unique decomposition of any element $x$ such that $x-ab$ is not defined, into a sum of a multiple of $a$ and a multiple of $b$.
    $$\forall x\exists y \exists z[\neg \exists w (w+ab=x) \Rightarrow (M_a(y)\land M_b(z)\land x=y+z)].$$
     \item Definition of $\alpha,\beta,\pi_a$, and $\pi_b$ partial functions (as per Proposition \ref{basic facts}).

    \item Behavior of $R_{n,r}$ predicates on the constants. 
 \begin{itemize}
\item $(R_{n,0}(a) \Rightarrow \neg R_{n,0}(b))\land (R_{n,0}(b)\Rightarrow \neg R_{n,0}(a))$ for all $n\in\N$.
 \item $(R_{n,i}(a)\land R_{n,j}(b))\Rightarrow R_{n,ij}(ab)$ for all $n\in\N$.
   \item $\forall x[ x\geq c \Rightarrow s^1(x)=x+\beta_1-\alpha_1]$.
 \item $R_{n, i}^a(ab)\iff R_{n, i}(b)$.
  \item $R_{n, i}^b(ab)\iff R_{n, i}(a)$.
 \end{itemize}

    \item For each $n,r\in \N$ we add 
    $\forall x [R_{n,r}(x)\iff \exists y(n(ab)+x=n(y)+r)]$ where adding $r$ just means iterating the successor function $r$ times. 
     
\end{enumerate}

    \end{dfn}

\begin{rem}\label{implication of residues}
    Notice that \[\forall x [M_a(x) \wedge R_{n,i}^a(x)\wedge R_{n,j}(a)\Rightarrow R_{n,ij}(x)]\] and 
    \[\forall x [M_b(x) \wedge R_{n,i}^b(x)\wedge R_{n,j}(b)\Rightarrow R_{n,ij}(x)]\] both follow from the divisibility (by $n$) conditions implied by  $R_{n,i}, R_{n,i}^a$ and $R_{n,i}^b$.
\end{rem}

\begin{prop}\label{Predicates Behavior}
  Addition, uniqueness of the residue, and the Chinese Remainder Theorem all hold for $R_{n,r}^a, R_{n,r}^b$ and $R_{n,r}$. That is to say, the following hold,   

\begin{enumerate}
    \item  For all $n\in \mathbb N$ we have $R_{n, i}(x)\wedge R_{n, j}(y)\Rightarrow R_{n, i+j} (x+y)$ for any $x,y$.
\item     $\forall x\left[ x\geq c \Rightarrow \bigvee_{i=0}^{n-1}\left(R_{n,0}(s^i(x))\land \bigwedge_{j\neq i}\neg R_{n,0}(s^j(x))\right)\right]$
 \item    $\forall x[(R_{n,i}(x)\land R_{m,j}(x))\Rightarrow R_{nm,k}(x)]$
    for $n,m$ relatively prime, and where $k$ is the unique solution modulo $nm$ to the congruences $x\equiv_n i,x\equiv_m j$
\end{enumerate}
as well as the corresponding statements for the $M_a$- and $M_b$-residues.

\end{prop}
\begin{proof}
The proofs for the $M_a$- and $M_b$-residues follow from the definitions of the predicates given in axiom $5.$  of $T_{ons}$, following the standard proofs (as found for example in \cite{NiZu}). 

Since the definition of $R_{n,i}$ is not standard (since we had to move past the conductor to ensure divisibility) the standard proofs have to be modified as follows:

For all $x$, $R_{n,i}(x)$ holds if and only if $R_{n,i}(nab+x)$ holds: Assume $nab+(nab+x)=ny+i$ for some $y$. We may assume $x>i$, so that in particular $y>2ab$, which implies that $ab+w=y$ for some $w$ (this is the defining property of the conductor). It follows that 
    $nab+x=nw+i$. For the converse, notice that $nab+x=nw+i$ always implies $nab+nab+x=n(w+ab)+i$.

It follows that we can prove any of the statements replacing $x$ by $nab+x$. However, the statements above show that for any element $x$, if $nab$ can be subtracted from $x$, then $R_{n,i}(x)$ holds if and only if $x=nw+i$ for some $w\in S$. So for these elements, $R_{n,i}$ implies the standard definition of residues and everything can be proved as in Chapter 3 Section 1 in \cite{Mar}.  
\end{proof}

\begin{rem}\label{incompleteness}
    Any standard 2-semigroup is a model of $T_{ons}$ with the canonical interpretations of the constants, predicates and functions, as per Proposition \ref{basic facts} and Remark \ref{presburger}. Hence, $T_{ons}$ is not a complete theory, as different generators may have different residue predicates.
\end{rem}

\begin{rem}\label{r-residue determined}
    In a 2-semigroup $S$, any element $x\in S$ with unique decomposition has its $R$-residue type completely determined by the $M_a$-residue type of $\pi_a(x)$, the $M_b$-residue type of $\pi_b(x)$, and the $R$-residue types of $a$ and $b$. This is a direct consequence of Remark \ref{implication of residues} and Proposition \ref{Predicates Behavior}.
\end{rem}

The following theorem will be quite useful.

\begin{theo}\label{Axiom 12}
    Let $M$ be a $\hat{\mathcal L}$-structure satisfying conditions 1 through 10 in Definition \ref{dfn axioms} and let $n\in \N$. Then for any $\beta\in M_b$, and $\alpha\in M_a$  such that $res_n(\beta)=\res_n(\alpha)$, there is some $\beta'\in M_b$ and $\alpha'\in M_a$ such that 
       \[
       \beta-\alpha=n(\beta'-\alpha').
       \]

       In particular, if we assume additionally that $\beta-\alpha<na$ then $\beta-\alpha=n(\beta'-\alpha(\beta'))$ for some $\beta'$.
\end{theo}

\begin{proof}
  To prove the first item we need to show that $\bigcup_{n\in \N} n(ab)+M_a$ is a model of Presburger. More precisely, we need: 
  \begin{claim}
Let $s:=\res_n^a(ab)= \res_n(b)$, let $\alpha\in M_a$ with $\res_n^a(\alpha)=r$ and let $k\in \N$ with $k<n$ be such that $n|r+ks$. Then there is some $y\in M_a$ such that $n(y)=\alpha+k(ab)$.

Similarly, if  $s:= \res_n^b(ab)= \res_n(a)$, $\beta\in M_b$ with $ \res_n^b(\beta)=r$ and $k\in \N, k<n$ is such that $n|r+ks$, then there is some $x\in M_b$ such that $n(x)=\beta+k(ab)$.
  \end{claim}
\begin{claimproof}
    Because $M_a$ is a model of bounded Presburger, we know that there is some $y_\alpha\in M_a$ such that $n(y_\alpha)=\alpha-r(a)$. Similarly, there is some $y_{ab}$ such that $n(y_{ab})=ab-s(a)$. 
So
    \[
    n(y_\alpha+k(y_{ab}))=\alpha-r(a)+k(ab)-ks(a)=\alpha+k(ab)-(r+ks)(a).
    \]

Since $n$ divides $r+ks$ by hypothesis, $\frac{r+ks}{n}\in \N$ and $y:=y_\alpha+k(y_{ab})+\frac{r+ks}{n}(a)$ will satisfy the first statement in the claim.

The second statement is analogous.
\end{claimproof}

We will now prove the first item, so assume that $\alpha\in M_a$, $\beta\in M_b$ are such that $\res_n(\beta)=\res_n(\alpha)$. By induction (or well ordering of $\N$), it is enough to prove the statement for $n=p$ prime.

By condition 10 from Definition \ref{dfn axioms}, $p$ cannot divide both $a$ and $b$. Assume that $p$ does not divide $b$ (the other case is analogous). Since $ \res_p^a(ab)= \res_p(b)\neq 0$, we know that if $r=\res_p^a(\alpha)$ and $s=\res_p^a(ab)$ there is some $k<p$ such that $p| r+ks$. By the claim, there is some $y_\alpha\in M_a$ such that $p(y_\alpha)=\alpha+k(ab)$. 

This implies $R_{p,0}(\alpha+k(ab))$ and by hypothesis $R_{p,0}(\beta+k(ab))$. Since $p\not| b$ we have $\neg R_{p,0}(b)$ which by definition and the claim above (as in Remark \ref{implication of residues}) implies $R_{p,0}^b(\beta+k(ab))$. So there is some $x_\beta\in M_b$ such that $p(x_\beta)=\beta+k(ab)$. By construction, $\beta-\alpha=p(x_\beta-y_\alpha)$, as required.

\medskip

For the ``in particular'' statement, if $\beta-\alpha<na$ and $\beta-\alpha=n(\beta'-\alpha')$ then $\beta'-\alpha'<a$ and by definition $\alpha(\beta')=\alpha'$.
\end{proof}

\begin{rem}
If $v-w=\frac{x-y}{n}$ then $n(v)+y=n(w)+x$ and the $R$-residue modulo $n$ of $x$ and $y$ are the same. So the above theorem is really an equivalence.\end{rem}

\section{Theories of Limit 2-Semigroups and their invariants}\label{Section invariants}

Now we formally define limit 2-semigroups, as we wish to explore their theory.

\begin{dfn}
    Let $\mathcal{U}$ be a non-principal ultrafilter on $\N$ and let $(S_i)_{i \in \N}$ be a sequence of standard 2-semigroups. A limit ordered numerical 2-semigroup (henceforth limit 2-semigroup) is the ultraproduct $S := \prod_{\mathcal{U}} S_i$.
    
    If $S$ is isomorphic to a standard semigroup, we call it trivial. Henceforth we shall assume all limit 2-semigroups are nontrivial.
\end{dfn}
\begin{dfn}\label{limit semigroup}
    If $S = \prod_{\mathcal{U}} S_i$ is a limit 2-semigroup, and each $S_i$ is generated by $(a_i, b_i)$, then we call $a := [(a_i)_{i\in \N}]_\mathcal{U}$ and $b:= [(b_i)_{i\in \N}]_\mathcal{U}$ the \emph{generators} of $S$.
\end{dfn}

\begin{rem}
    The generators of a limit 2-semigroup are not generators in the classical sense: not all elements of $S$ can be written as $\N$-linear combinations of $a$ and $b$.
\end{rem}

We will restrict ourselves to 2-semigroups with non-finite generators. This is of course consistent with the idea mentioned in the introduction of taking the limits of 2-semigroups $S_n$, where $S_n$ is a semigroup obtained by choosing two random integers in $[1, n]$, dividing both by their greatest common divisor, and taking the lesser of the resulting two numbers as the generator $a_n$ of $S_n$ and the greater as $b_n$, and then taking their ultralimit. 

So we will add an axiom implying that $a$ is not finite, meaning that it is greater than any finite multiple of the least possible non-zero distance between elements of the semigroup (since $a<b$, the fact that $b$ is not finite follows). More precisely:

\begin{dfn}
    The $\hat{\lang}$-theory $T_{lons}$ is the extension of $T_{ons}$ that includes the axioms
        $$a+k\alpha_1> k\beta_1~~~\text{(for all natural numbers}~k)$$
     \end{dfn}

\subsection{Ratios and the order type of the constants}\label{Invariants}
Since the language $\hat{\lang}$ includes an order relation, as well as $+$, $-$, and $\frac{\cdot}{n}$, any complete theory will need to decide the order type of the elements generated by the constants.

Let us consider first the closure of the constants $a,b,ab$ under the operations $C_0:=\langle a, b, ab\rangle_{\{+,-,\frac{\cdot}{n}\}}$. Then $M_a(C_0)$ is the union of $\{n(a)\}_{n\in \N^\times}$,  $\left\{\frac{ab-res_n^a(ab)}{n}+k(a)\right\}_{n\in \N^\times, k\in \Z}$ and $\{ab-n(a)\}_{n\in \N}$, and similarly for $M_b(C_0)$. The order within each $M_a(C_0)$ and $M_b(C_0)$ is set by the fact that $ab>n(b)$ for all $n\in \N$ (this follows from the definition of conductor, unique decomposition of elements under $ab$ and the fact that $a$ is infinite). This fact also implies that the order in 
\[C_0=\bigcup_{n\in \N} n(ab)+M_a(C_0)+M_b(C_0)\] will be implied by the order in $M_a(C_0)+M_b(C_0)$ which is determined by understanding when $n(a)<m(b)$ for $m,n\in \N$ or more precisely, by the set
\[
\left\{\frac{m}{n}\mid n(a)<m(b) \right\}.
\]

This prompts the following definition.

\begin{dfn}
    Let $S$ be a limit 2-semigroup, and let $x,x',y,y'$ be $\emptyset$-definable elements of $S$, such that $x\geq x'$ and $y>y'$. We define the ratio between $x-x'$, and $y-y'$, as the real number (or infinity)
    $$r(x-x',y-y'):=\sup\left\{\frac{p}{q}~\bigg|~p,q\in\Z^{\geq 0}, q\neq 0, S\models py+qx'\leq qx+py'\right\}$$
    When $x-x'=x_0\in S$, we will abbreviate such ratio as $r(x_0,y-y')$. Analogously if $y-y'$ exists.
\end{dfn}

The ratio between $x-x'$ and $y-y'$ can be thought of as the supremum of the rational numbers $\frac{p}{q}$ such that $\frac{p}{q}\leq \frac{x-x'}{y-y'}$ (regardless of whether $x-x'$ and $y-y'$ are defined).

Of course, $r(x-x', y-y')$ is not a formula. However, $r_{p/q}(x-x', y-y'):=py+qx'\leq qx+py'$ is. This implies for example that if in some structure $S$ we have $r(b-\alpha(b),a)=1/\pi$, then 
\[
S\models p(b)\leq q(a)+p(\alpha(b))
\]
if and only if $p/q<1/\pi$,
which implies that in any structure $M$ satisfying the same theory as $S$, the generators $a_M$ and $b_M$ in $M$ also satisfy 
\[
M \models p(b_M)\leq q(a_M)+p(\alpha(b_M))
\]
if and only if $p/q<1/\pi$, so that $r(\alpha(b_M)-b_M, a_M)=1/\pi$. In this sense, the ratios of definable elements are invariants of the theory of a limit 2-semigroup. 

\medskip

With this notation, the fact that the order restricted to $C_0$ is determined by $\left\{\frac{m}{n}\mid n(a)<m(b) \right\}$ implies that we need to fix $r(a,b)$ in order to determine the complete theory. Conversely, if $r(a,b)$ is 0, or irrational, then $r(a,b)$ uniquely determines $\left\{\frac{m}{n}\mid n(a)<m(b) \right\}$ which in turn implies the order in $C_0$. We let $q_0:=r(a,b)$.

\bigskip

If $q_0= 0$, then $\alpha(b)\not\in M_a(C_0)$ (in fact, $q_0= 0$ if and only if $\alpha(b)\not\in M_a(C_0)$). In order to close under $\alpha$ we need to add at least $\alpha(b)$. Turns out that $\langle a, b, ab\rangle_{\{+,-,\frac{\cdot}{n}, \alpha\}}=\langle a, b, ab, \alpha(b)\rangle_{\{+,-,\frac{\cdot}{n}\}}$: For any $\beta_1, \beta_2\in M_b(C_0)$ we have that either $\alpha(\beta_1+\beta_2)=\alpha(\beta_1)+\alpha(\beta_2)$ or  $\alpha(\beta_1+\beta_2)=\alpha(\beta_1)+\alpha(\beta_2)+a$ depending only on whether or not 
\[
\beta_1+\beta_2-(\alpha(\beta_1)+\alpha(\beta_2))<a
\]
or not, which in turn is determined uniquely by \[\left\{\frac{p}{q}~\bigg|~p,q\in\N, q\neq 0, \  p(\beta_1)\leq q(a)+p(\alpha(\beta_1))\right\}\] and \[\left\{\frac{p}{q}~\bigg|~p,q\in\N, q\neq 0, \  p(\beta_2)\leq q(a)+p(\alpha(\beta_2))\right\}.\] 

This implies two things. First, that $\langle a, b, ab\rangle_{\{+,-,\frac{\cdot}{n}, \alpha\}}=\langle a, b, ab, \alpha(b)\rangle_{\{+,-,\frac{\cdot}{n}\}}$. And second, that the structure and order in $\langle a, b, ab\rangle_{\{+,-,\frac{\cdot}{n}, \alpha\}}$ is determined uniquely by $\left\{\frac{p}{q}~\bigg|~p,q\in\N, q\neq 0, \  p(b)\leq q(a)+p(\alpha(b))\right\}$ which, will in turn be determined uniquely by $q_1:=r(b-\alpha(b),a)$ if this happens to be irrational or zero.

\bigskip

Finally, we will also need to study the order between definable elements once we include the constant $\beta_1$. If $n(b)-\alpha(n(b))=m(\beta_1-\alpha_1)$ for some $n,m\in \N$, then $m(\beta_1-\alpha_1)\in \langle a, b, ab\rangle_{\{+,-,\frac{\cdot}{n}, \alpha\}}$ and by Theorem \ref{Axiom 12} and unique decomposition, $\beta_1\in \langle a, b, ab\rangle_{\{+,-,\frac{\cdot}{n}, \alpha\}}$. This is in fact a necessary and sufficient condition, so if $n(b)-\alpha(n(b))$ is never equal to $m(\beta_1-\alpha_1)$, we need to add $\beta_1$ in order to understand the theory on the constants.

Notice that since $\beta_1$ and $ab$ are in $M_b$, the fact that $n(\beta_1)$ is subtractable or not from $m(ab)$ depends on whether or not $n(\beta_1)<m(ab)$. So we need to understand $\left\{\frac{p}{q}~\bigg|~p,q\in\N, q\neq 0, \  p(\beta_1)\leq q(ab)\right\}$ which will be completely implied by $r(\beta_1, ab)$ if the later is irrational (see Remark \ref{betas ratios}).

\medskip

To summarize, in order to get a complete theory we need to add to $T_{lons}$ the following invariants. Any combination of these will give rise (if it is consistent) to a different theory of a limit 2-semigroup.

\begin{Inv}\label{all invariants}
Consider the following.
\begin{itemize}
\item All the congruences of $a,b$ modulo any $p^n$ with $p,n \in \N$ and $p$ prime. These congruences must of course be compatible with the assumption that $gcd(a,b)=1$. Similarly, we will need the $b$-congruences of $\beta_1$ and the $a$-congruences of $\alpha(b)$ and $\alpha_1$.

\item Statements of the form $ma<nb$. If $q_0:=r(a,b)$ is irrational or zero, these are implied by the axiom scheme implying $r(a, b)=q_0$.

\item If $q_0\neq 0$ we need to add either $n(b)-\alpha(n(b))=m(\beta_1-\alpha_1)$ if this holds for some $n,m\in \N\setminus\{0\}$, or all the statements $m(b)<n(\alpha(b))+n(a)$ that hold in the structure with $m,n\in \N$. Once again, these are implied by the axiom scheme $r(b-\alpha(b), a)=q_1$ if $q_1:=r(b-\alpha(b), a)$ is irrational.

\item If $n(b)-\alpha(n(b))\neq m(\beta_1-\alpha_1)$ for all $m,n\in \N^\times$, then we need to add all statements $m(\beta_1)<n(ab)$. If $q_2:=r(\beta_1, ab)$ is irrational or zero, this is implied by the axiom scheme $r(\beta_1, ab)=q_2$.
\end{itemize}
\end{Inv}

\begin{rem}\label{betas ratios}
    For $n\in\N$, the element $\beta_n$ is defined as the unique multiple of $b$ that is $n$ units greater than a multiple of $a$, where a unit is the least possible nonzero distance between two elements of the semigroup (Proposition \ref{basic facts} and Remark \ref{alphas and betas}). When $q_2=r(\beta_1,ab)$ takes an irrational value, the ratio $r(\beta_n,ab)$ is the decimal part of the number $nq_2$, as $\beta_n=n(\beta_1)-k_n(ab)$, where $k_n(ab)<n\beta_1<(k_n+1)(ab)$. It's worth noting that for rational values of $q_2$, $r(\beta_n,ab)$ is not necessarily the decimal part of $nq_2$, since for instance $q_2=1/2$ can happen with $2(\beta_1)<ab$ (which results in $r(\beta_2,ab)=1$) or with $2(\beta_1)>ab$ (which results in $r(\beta_2,ab)=0$).
\end{rem}

\subsection{Compatibility of invariants.}
Not all choices of invariants are compatible between them. For example, if $q_0\neq 0$ and $n$ is such that $1/n\leq q_0<1/(n+1)$, either $\alpha(b)=n(a)$ or $b<n(a)$, $\alpha(b)=(n-1)(a)$ and $q_0=1/n$. 

In the former case $q_1=r(b-na, a)=1/q_0-n$ and in the latter $q_1=1$. Either way $q_1\neq 0$ whenever $q_0\neq 0$ is different from $1/n$ for some $n\in \N$ which in particular implies that $b\neq \beta_1$ whenever $q_0$ is irrational. 

Similarly, if $mb=\beta_n$ for some $m,n\in \N$ then either $q_0=0$ or $q_0=1/k$ for some $k$.

\medskip

In this subsection we will address whether there are other restrictions for choices of the invariants to be witnessed by limit $2$-semigroups. For simplicity, we will only focus on the compatibility of congruences and the different ratios $q_0, q_1, q_2$. We will also pay close attention at the case where $mb=\beta_n$ for some $m,n\in \Z$ (which implies $q_2=k/n$ for $k$ as in Theorem \ref{Axiom 12} with $\beta=\beta_n$ and $\alpha=\alpha_n$), since this is the only case where we will be able to prove a complete axiomatization of the theory.

\medskip

\L o\'s' Theorem implies that if we fix some invariants as in Subsection \ref{Invariants}, namely, congruences for $a$ and $b$, $M_a$-congruences for $\alpha_1$ and $\alpha(b)$, the $M_b$-congruence of $\beta_1$, and the ratios $q_0, q_1, q_2$ of $r(a,b), r(b-\alpha(b), a)$ and $r(\beta_1, ab)$, the existence of a limit 2-semigroup with these ratios happens if and only if we can approximate the values by finite 2-semigroups.

Namely, whether for any fixed $n\in \mathbb N$ and any $i_0, i_1, i_2$ less than $n$, any complete set of congruences for $a,b$ and $\alpha_1$ modulo $n$ together with the statements $\frac{i_0}{n}\leq \frac{a}{b}<\frac{i_0+1}{n}$, $\frac{i_1}{n}\leq \frac{b-\alpha(b)}{a}<\frac{i_1+1}{n}$ and $\frac{i_2}{n}\leq \frac{\beta_1}{ab}<\frac{i_2+1}{n}$ can be realized in a \emph{finite} 2-semigroup. This is of course a question about standard semigroups, and we will not be able to answer it in this paper, but we will give some evidence about this.

Fix any semigroup generated by elements $a,b\in \N$ and let $l$ be the residue of $b$ modulo $a$. 

$b-(b-\alpha(b))$ is $\alpha(b)$, which is an actual multiple of $a$. So $b-\alpha(b)=l$ and $q_1=r(b-\alpha(b),a)=l/a$.

By definition, $\beta_1$ is a multiple of $b$, which is equal to $1+\alpha_1$, with $\alpha_1$ a multiple of $a$. This implies that $\beta_1\cong_a 1$. So in any finite 2-semigroup, if $\beta_1=kb$, we have that $k$ is the multiplicative inverse of $l$ modulo $a$ (since $b\cong_a l$), and $q_2=r(\beta_1, ab)=r(kb, ab)=r(k,a)$. 

This implies that in order to approximate any pair of real numbers $q_1$ and $q_2$ as the corresponding ratios, we need to show that the possible pairs $(r, r^{-1})$ of invertible residues modulo $a$ normalized by $a$ is dense in $[0,1]^2$. This was done in \cite{BeKh}: the authors proved that the set of points $(r/a, r^{-1}/a)$ is evenly distributed on the square $[0,1]\times [0,1]$. 

By \L o\'s' Theorem, this even distribution proved in \cite{BeKh} implies:

\begin{prop}
    Given any real numbers $0\leq q_1, q_2\leq 1$, there is a limit 2-semigroup with generators $a,b$ such that $r(b-\alpha(b), a)=q_1$ and $r(\beta_1, ab)=q_2$.
\end{prop}
 
\bigskip

\medskip

This of course does not take into account the congruences. We plotted some of the possible values of $q_1$ and $q_2$ fixing the congruences of $a$ and $b$ modulo some fixed $n$. Figures 1, 2 and 3 show the results:

\begin{figure}[h]
\centering
    {\includegraphics[width=10cm]{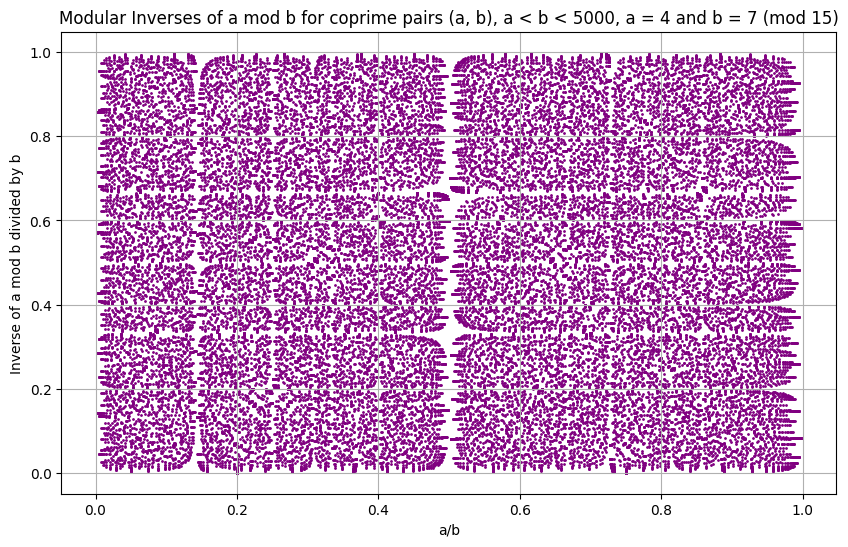} }%
    \caption{Example plot with $N \leq 5000$, $a= 4, b=7$ (mod 15)}
\end{figure}

\begin{figure}[h]
\centering
    \subfloat[\centering $N\leq 200$]{{\includegraphics[width=3.2cm]{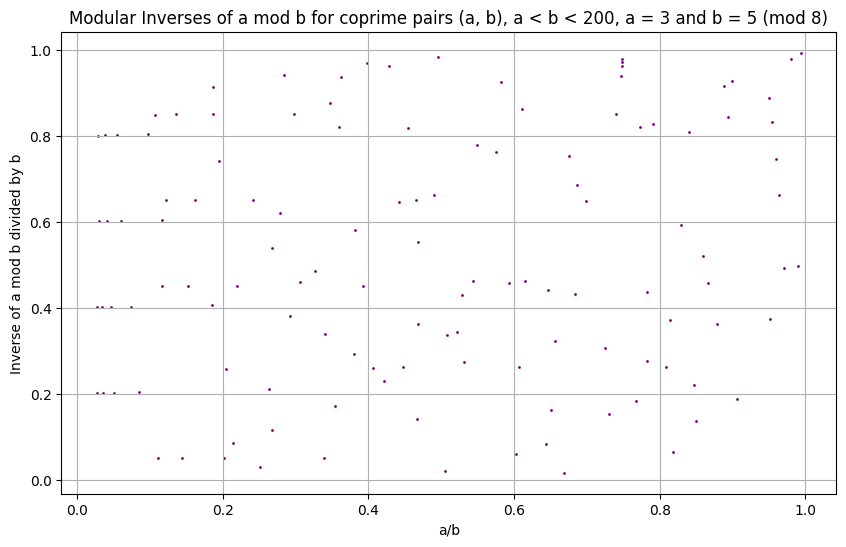} }}%
    \qquad
    \subfloat[\centering $N\leq 1000$]{{\includegraphics[width=3.2cm]{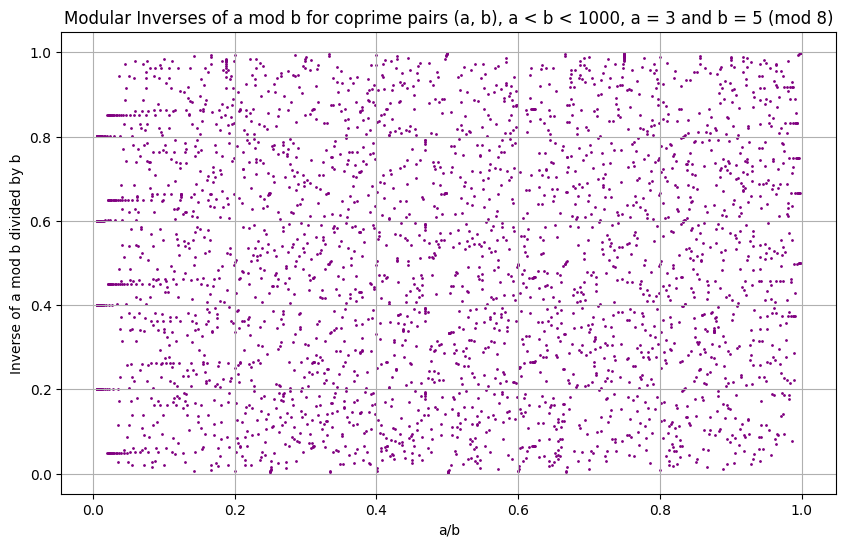} }}%
    \qquad
    \subfloat[\centering $N\leq 5000$]{{\includegraphics[width=3.2cm]{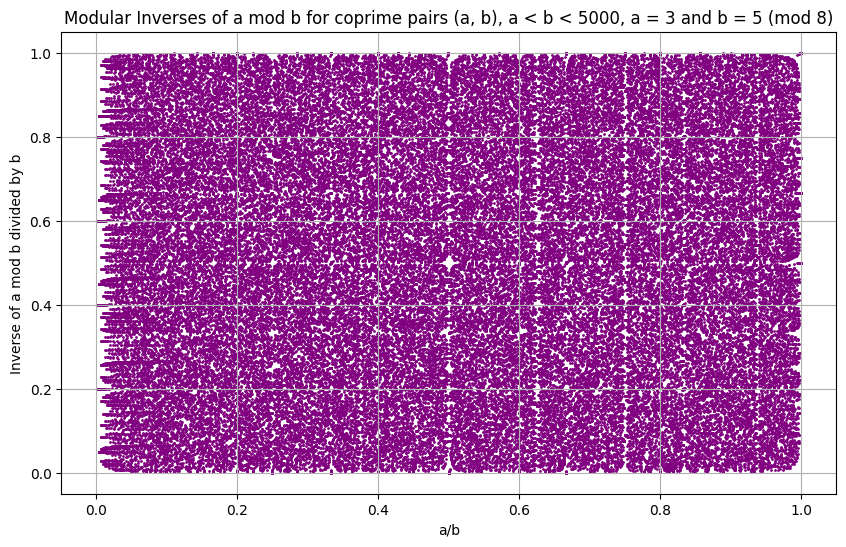} }}%
    \caption{Examples with varying $N$ and $a=3, b= 5$ (mod 8), showing the increase in density as $N$ increases} %
\end{figure}

\begin{figure}[h]
\centering
    \subfloat[\centering $a =4, b =7$ (mod 60)]{{\includegraphics[width=3.2cm]{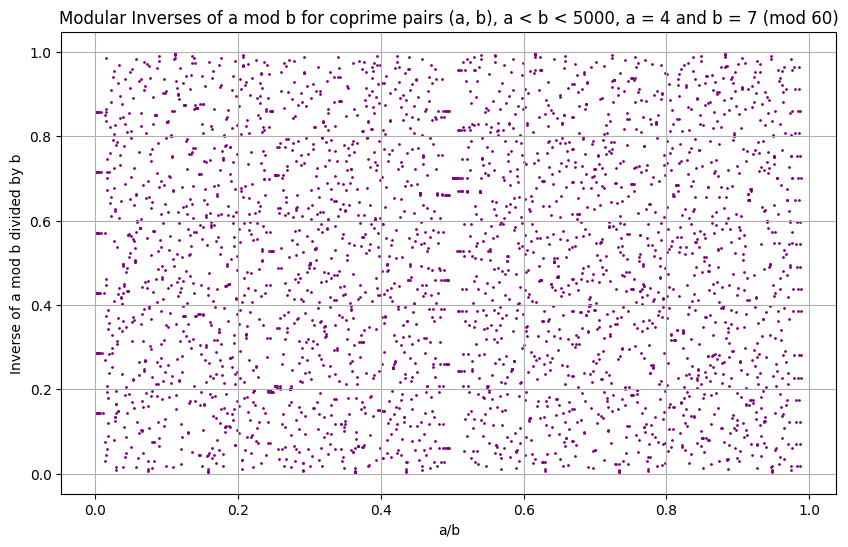} }}%
    \qquad
    \subfloat[\centering $a =17, b =41$ (mod 60)]{{\includegraphics[width=3.2cm]{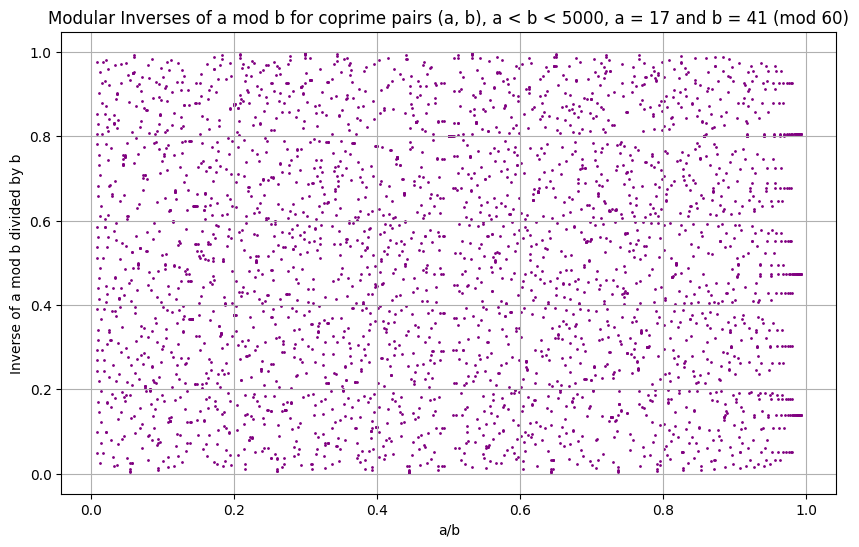} }}%
    \qquad
    \subfloat[\centering $a =20, b =33$ (mod 60)]{{\includegraphics[width=3.2cm]{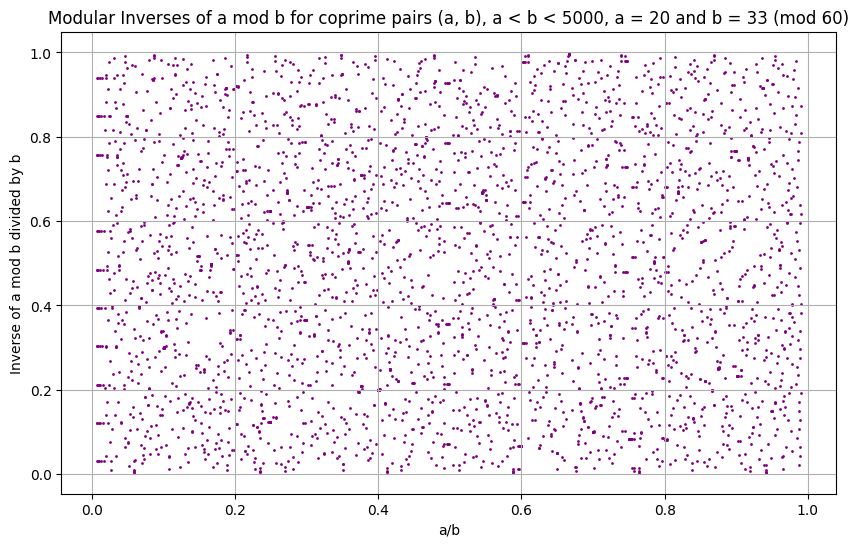} }}%
    \caption{Examples with $N \leq 5000$ and several choices of mod 60 residues for $a$ and $b$, showing no apparent differences in behavior. Note the mod 60 residues must be chosen to be coprime.} %
\end{figure}

It appears that any choice of $q_1$ and $q_2$ is compatible with any congruence choice of $a$ and $b$. But proving this would imply proving even distributions such as in \cite{BeKh}, but restricting the congruences of $a$ and $b$ modulo any number $n$. This is a very interesting question, but beyond what we are doing in this paper.

\section{Model theoretic considerations}\label{Section: model theoretic}

In the following section, we will provide necessary conditions for the completeness of the theory $T_{lons}$ once we fix the ratios $q_0, q_1$, and $q_2$. We will then use these to prove completeness of the theory for a very particular case. 

It is worth giving an intuition as to what we are going to do. We will rely on a strategy to show completeness called \emph{quantifier elimination}. This, essentially, states that if a given set of axioms implies whether or not a system of equations has a solution or not, then one is quite close to having a complete set of axioms. What a ``system of equations'' is depends on the theory. If we have quantifier elimination \emph{and} there is an initial object among the models of $T$, then the theory is complete. More precisely, the following is Theorem 3.1.15 in \cite{Mar}.

\begin{fact}\label{Marker}
If a theory $T$ has elimination of quantifiers and there is a model $M_0\models T$ that embeds into every model of $T$, then $T$ is complete.
\end{fact}
The model $M_0$ above is called \emph{a prime model of $T$}.

In this section we will explore what is needed for quantifier elimination, explain why we will restrict ourselves to the case $m(b)=\beta_n$, and prove that in this case the $\hat\lang$-structure generated by the constants is a prime model of the structure.

\subsection{Systems of equations and the theory of $\mathcal Z_a$}

We will prove (modulo a small caveat, which we will get into in the following section) 
that for our system of axioms, in order to prove completeness, we need to show that given any 2-semigroup $S$, the axioms we set for $S$ (meaning $T_{lons}$ and the invariants satisfied by $S$) are enough to decide, for any $b_1, b_2\in M_b(S)$, any two  definable distances $a_1,a_2$ less than $a$, and any $n\in \N$,  whether or not there is an element $x$ in $M_b(S)$ satisfying: 
\begin{itemize}
\item A (fixed) complete set of $M_b$-residues for $x$ modulo $n$.
\item A complete set of residues modulo $n$ for $\alpha(x)$. Given the previous item, it is equivalent to giving a complete set of residues modulo $n$ for $x-\alpha(x)$.
\item $a_1<x-\alpha(x)<a_2$.
\item $b_1< x<b_2$.
\item A complete set of $M_a$-residues modulo $n$ for $\alpha(x)$.
\end{itemize}

\begin{dfn}
    We will refer to any combination as above as a \emph{system of equations} (or an $n$-\emph{system of equations} if $n$ is important). 
\end{dfn}

We will say that a 2-semigroup (limit or standard) \emph{realizes} a system or equations $\Phi(x)$ if there is an element in the 2-semigroup satisfying the system. Otherwise we will say that the 2-semigroup \emph{omits} $\Phi(x)$. 

\medskip

The $M_a$-residues modulo $n$ of $\alpha(x)$ fully determine its residues modulo $n$, as Remark \ref{r-residue determined} states. However, in certain cases, the $R$-residues of $\alpha(x)$ determine its $M_a$-residues:

\begin{rem}\label{finite primes}
    If $p^n$ is the maximum power of $p$ dividing $a$, then a complete set of residues for $\alpha(x)$ modulo $p^{n+k}$ (which in particular must imply that $\alpha(x)$ is divisible by $p^n$) will determine a complete set of $M_a$-residues modulo $p^k$. This implies that whenever $a$ is such that for every prime $p$ there is some $n$ such that $\neg R_{p^n,0}(a)$, the complete set of $M_a$-residues for $\alpha(x)$ can be determined if we know a complete set of residues modulo $m$ for $\alpha(x)$ for some large enough $m$.
\end{rem}

In light of the remark above, we will define:

\begin{dfn}\label{def: reduced system}
    A \emph{reduced system of equations} is a list of formulas consisting of:
    \begin{itemize}
\item A complete set of $M_b$-residues for $x$ modulo $n$.
\item A complete set of residues modulo $n$ for $\alpha(x)$ (equivalently given the previous item, for $x-\alpha(x)$).
\item $a_1<x-\alpha(x)<a_2$.
\item $b_1< x<b_2$.
\end{itemize}
\end{dfn}

To help with the notation, let $\phi(x)$ be the function assigning to each element $x\in M_b$ the element $x-\alpha(x)$ (which in the standard case is just $\res_a(x)$). This operation is clear for any finite semigroup $S_{fin}$ and as we saw in the previous section, is interpretable in any limit 2-semigroup $S$.

We will fix a limit 2-semigroup generated by elements $a,b$. As mentioned before, understanding whether or not our axioms imply a system of equations, is equivalent to showing that given any system of equations $\Phi(x)$, we can find a finite subset of our axioms $\Sigma$ such that either every (standard) 2-semigroup satisfying $\Sigma$ realizes $\phi(x)$ or every (standard) 2-semigroup satisfying $\Sigma$ omits $\phi(x)$.

\begin{rem}\label{follows discussion}
   A complete set of axioms must decide whether or not a system is realized. What we will prove in the following section is that if the axioms decide whether or not any particular system is realized, then one has quantifier elimination which for certain $q_1, q_2$ will imply completeness.
\end{rem}

For any finite $n\in \mathbb N$, let $(res_n, +_n, <)$ be the structure consisting of the natural numbers less than $n$ together with addition modulo $n$ and the order.

Similarly, we will refer as $(res_a, +_a, <)$ to all the definable distances less than $a$ with addition modulo $a$, and the order between them, and as $(M_b(S), +_{ab}, <)$ to the elements in $M_b(S)$ with addition modulo $ab$.  

For finite $a,b$, the structure $(res_{a}, +_{a}, <)$  is naturally isomorphic to $(M_{b}(S), +_{ab}, <)$ with the isomorphism that sends $w<a$ to $bw$. Notice that the $M_b$-residues of $bw$ modulo $n$ is just the residue of $w$ modulo $n$.

We also have a map from $(M_b(S), +_{ab}, <)$ to $(res_{a}, +_{a}, <)$, that sends $x$ to $\phi(x)$. In the standard case $\phi(x)=\res_a(x)$, so composing with the isomorphism from $(res_{a}, +_{a}, <)$ to $(M_{b}(S), +_{ab}, <)$ we get a map from $(res_{a}, +_{a}, <)$ to itself sending $w$ to $\res_a(bw)=\res_a(b)\odot_a\res_a(w)$. So we are just multiplying $w$ by the residue of $b$ modulo $a$. Let $\lambda$ be such a map.

\medskip

Fixing $x=bw$ with $w<a$, understanding whether or not an $n$-system of equations is realized by $x$ is equivalent to understanding whether or not there is some $w$ such that, for some fixed $n\in \N$ and some $a_1, a_2$ and $d_1, d_2$ (the images of $b_1$ and $b_2$ in $M_b(S)$ under the aforementioned isomorphism) all distances less than $a$,
\begin{itemize}
\item $w$ satisfies a complete set of residues modulo $n$
\item $\lambda(w)$ satisfies a complete set of residues modulo $n$.
\item $a_1< \lambda(w)<a_2$.
\item $d_1< w <d_2$.
\end{itemize}

These are all questions about $(res_{a}, +_{a}, <_{a}, \lambda)$.

\medskip

If $b$ is not finite, we don't have an isomorphism from  $(res_{a}, +_{a}, <)$ to $(M_{b}(S), +_{ab}, <)$ (in the finite case it is just adding everything $b$ times). And multiplication cannot be added to the language with any hope of succeeding since the resulting theory would be undecidable. This implies that the structure $(res_{a}, +_{a}, <_{a}, \lambda)$ is not necessarily definable in the limit 2-semigroup $S$ with generators $a,b$, but it is close enough:  $(res_{a}, +_{a}, <_{a})$ is, and  $\lambda$ as a definable map between the two disjoint copies of $(res_{a}, +_{a}, <_{a})$ definable in $S$, namely, $(M_{b}, +_{ab}, <_{ab})$ and $(res_{a}, +_{a}, <_{a})$.

And an axiomatization for $(res_{a}, +_{a}, <_{a}, \lambda)$ is stronger than what we need. If we understood this structure we would be able to understand whether a reduced system can be solved.

By Fact \ref{Garcia}, without the scalar multiplication $\lambda$  the theory of $(res_{a}, +_{a}, <_{a})$ (as the substructure of $S$ consisting on differences less than $a$) is decidable, axiomatizable, and the theory follows from the axioms $T_{lons}$ together with the invariants. If $r$ is ``rational'' (meaning a finite quotient of a finite sum of the first non-zero element in $(res_{a}, +_{a}, <_{a})$) and $\lambda_r$ is multiplication by $r$, then $(res_{a}, +_{a}, <_{a}, \lambda_r)$ is interpretable in $(res_{a}, +_{a}, <_{a})$, and whether or not a systems can be solved follows from the axioms. 

As we shall see in the next section this is enough to prove that \emph{in this case} the axioms and language we defined in Section \ref{Section 1} provide the complete theory of $S$ and give quantifier elimination.

\bigskip

To summarize, a complete axiomatization of the structure $S$ will involve a complete axiomatization of the structure consisting of two disjoint copies of $\left(res_{a}, +_{a}, <_{a}\right)$ together with the map $\lambda_{r}$ between them, and a complete axiomatization of $\left(res_{a}, +_{a}, <_{a}, \lambda_{r}\right)$ will imply whether a reduced system of equations is realized or not. 

However, as far as we can tell this is beyond what is known at the moment for arbitrary $r$. In the next section we will provide a complete axiomatization of $S$ for the case where we do understand the theory of $\left(res_{a}, +_{a}, <_{a}, \lambda_{r}\right)$ (i.e. $r$ rational), and prove that knowing the theory of $\prod(res_{a_i}, +_{a_i}, <_{a_i}), \lambda_{r_i}/\mathcal U$ is enough to axiomatize $S$ assuming that not prime ``divides'' $a$ ``infinitely many times''.

\subsection{Prime models}\label{Subsection Prime Models}
Let $T$ be the theory $T_{lons}$ together with the sentences describing a complete set of invariants as in Section \ref{Section invariants}. Let $\mathcal M\models T$. We will generate a substructure $M^0$ of $\mathcal M$ that is both a model of $T$ and such that it can be embedded into any other model of $T$. So all the elements we mention are assumed to be elements in $\mathcal M$.

We will define $M^0_a$ as the set of elements less than or equal to $ab$ which can be generated from $\{a, ab, \alpha(b), \alpha_1\}$ with addition, subtraction and division by $n$ restricted to elements in $Res^a_{n,0}(M_0^a)$ for any $n\in \mathbb N$. Recall that because it is a model of bounded Presburger, $y-x$ exists for $x,y\in M_a$ whenever $y>x$ and that $x\in M_a$ is divisible by $n$ whenever $R^a_{n,0}(x)$ holds. 

We define $M^0_b$ similarly, starting with $\{b, ab, \beta_1\}$ as generators. Similarly, $y-x$ exists for $x,y\in M_b$ whenever $y>x$ and $x\in M_b$ is divisible by $n$ whenever $R^b_{n,0}(x)$ holds.

Let $M^0:=\bigcup_{n\in \mathbb N} n(ab)+M^0_a+M^0_b$. 

If $\nodel$ is a different model of $T$, there is a natural injection $f$ from $M^0$ to $\nodel$ which just comes from first sending the constants in $\model$ to the constants in $\nodel$ and extending to injections from both $M^0_a$ and $M_b^0$ by closing under subtraction and division by $n$ and then closing under addition.  

Once we fix all the invariants, the functions $\pi_a, \pi_b, \alpha, \beta$, and the order in $M_0$ will coincide with the one in $f(M^0)$. We will not give all the details in the general case for two reasons: One,  essentially the proofs are a simplified version of the the proof of Lemmas \ref{inequalities} and \ref{localiso}. The other reason is that, because of the discussion in the previous subsection, we will only be able to prove quantifier elimination (and therefore completeness) in the case where $m(b)=\beta_n$ and $q_0=0$. In order to have a complete proof in this case and to give an insight into the more complicated proofs in the following section, let us explain why this holds. 

If $m(b)=\beta_n$, then $\beta_n\in \langle b, ab \clos$, and it follows from Theorem \ref{Axiom 12} that $\beta_1\in \langle b, ab\clos$, and 
\[M_b^0=\langle b, ab \clos=\{nb\}_{n\in \N}\cup \left\{\frac{l(ab)-\left(l\cdot \res_n^b(ab)\right)(b)}{n}+kb\right\}_{n>1, 0\leq l<n,  k\in \Z}\cup  \{ab-nb\}_{n\in \N}.\]
In this case we have an explicit isomorphism with $\mathcal Z_a$ by sending $nb$ to $n$, $\frac{ab-\res_n^b(ab)}{n}+ka$ to $\frac{a-\res_n(a)}{n}+k$ and $ab-na$ to $a-n$. Since $m(b)=\beta_n$, we have that $\phi(m(b))=n$, which under the identification implies that $\phi$ corresponds to the map from $\mathcal Z_a$ to $\mathcal Z_a$ sending $x$ to a $y$ such that $mx-ny=s(a)$ for some $s\in \N$. This map needs to be injective, so $(m, \res_m{a})=1$. This is completely determined by the theory of $([0, a), +_a, <)$, which by Fact \ref{Garcia} is completely determined by the set of residues of $a$ modulo $l$, where $l$ varies in $\N$.

This proves that

\begin{claim}
For any $x,y\in M_b^0$, $f(\alpha(x))=\alpha(f(x))$ and $x-\alpha(x)<y-\alpha(y)$ if and only if $f(x)-\alpha(f(x))<f(y)-\alpha(f(y))$.
\end{claim}

\begin{prop}
     $f$ is $\hat{\mathcal L}$-embedding (see Definition \ref{embedding-def}). 
\end{prop}

\begin{proof}
    We only need to prove that $f$ preserves the ordering. By construction, and since $ab>na$ and $ab>mb$ for all $m,n\in \N$, it is enough to show that if $\beta_1, \beta_2\in M_b^0$ and $\alpha_1, \alpha_2\in M_a^0$ then $\beta_1+\alpha_1<\beta_2+\alpha_2$ if and only if $f(\beta_1+\alpha_1)<f(\beta_2+\alpha_2)$. Assume $\beta_1+\alpha_1<\beta_2+\alpha_2$. There are two cases. If $\alpha(\beta_1)+\alpha_1<\alpha(\beta_2)+\beta_2$ then by definition of $\alpha$ we get $f(\beta_1+\alpha_1)<f(\beta_2+\alpha_2)$. If $\alpha(\beta_1)+\alpha_1=\alpha(\beta_2)+\beta_2$ then by the claim $\alpha(f(\beta_1))+f(\alpha_1)=\alpha(f(\beta_2))+f(\beta_2)$ and  $\beta_1+\alpha_1<\beta_2+\alpha_2$ if and only if 
\[\beta_1-\alpha(\beta_1)+\alpha(\beta_1)+\alpha_1<\beta_2-\alpha(\beta_2)+\alpha(\beta_2)+\alpha_2\] if and only if $\beta_1-\alpha(\beta_1)<\beta_2-\alpha(\beta_2)$. By the claim this implies $f(\beta_1)-\alpha(f(\beta_1))<f(\beta_2)-\alpha(f(\beta_2))$ which similarly implies $f(\beta_1)+f(\alpha_1)<f(\beta_2)+f(\alpha_2)$, as required.
\end{proof}

So in order to prove that $M^0$ is a prime model of $T$ the only thing we have left to prove is that $M^0$ is a model of $T$. By construction (and because it is a substructure of $\model$), it satisfies axioms $1$ through $10$ in Definition \ref{dfn axioms}. We only need to show $11$ holds as well.

\begin{prop}
    Let $M^0$ be the $\hat{\mathcal L}$-structure generated by the constants of any given model $\model$ as specified above. Let $n\in \N$. Then $R_{n,0}(x)\Rightarrow \exists y \ n(ab)+x=n(y)$.
\end{prop}

\begin{proof}
 By Theorem \ref{Axiom 12}, for any $\beta\in M_b^0$ and $\alpha\in M_a^0$, if $\res_n(\beta)=\res_n(\alpha)$ then there are $\beta'\in M_b^0$ and $\alpha'\in M_a^0$ wuch that $\beta-\alpha=n(\beta'-\alpha')$.

\bigskip

Now, let $x\in M^0$ be such that $R_{n,0}(x)$ holds. Let $x=m(ab)+\alpha+\beta$ for some $m$. It is clearly enough to prove that if 
$m>1$ there is some $y$ such that $n(y)=x$. 

Let $m=nl+k$ with $l,k\in \N, l\geq 1$ and $k$ the residue of $m$ modulo $n$. 

Let $s<n$ be such that $n$ divides $s+res^a_n(k(ab))+res^a_n(\alpha)$. By the claim in Theorem \ref{Axiom 12} there is some $y_\alpha\in M_a^0$ such that $n(y_\alpha)=s(a)+(k)ab+\alpha$ so in particular $R_{n,0}(s(a)+(k)ab+\alpha)$ and since  
\[0=res_n\left(m(ab)+\alpha+\beta\right)=res_n\left(nl(ab)+k(ab)+\alpha+s(a)+\beta-s(a)\right)\]
we have $res_n\left(k(ab)+\alpha+s(a))+\res_n(\beta)-\res_n(s(a))\right)=0$ and $\res_n(\beta)=\res_n(s(a))$. So there are $\beta'\in M_b^0$ and $\alpha'\in M_a^0$ such that $n(\beta'-\alpha')=\beta-s(a)$. 

But this implies that 
\[
n\left(\left(l-1\right)\left(ab\right)+\left(ab-\alpha'\right)+y_\alpha+\beta'\right)=nl(ab)-n(\alpha')+n(y_\alpha)+n(\beta')=x.
\]
Since $\left(l-1\right)\left(ab\right), \left(ab-\alpha'\right), y_\alpha$ and $\beta'$ are all in $M^0$, this completes the proof of the proposition.
  \end{proof}

\section{Quantifier Elimination}\label{Section quantifier elimination}
In this next section we will give necessary conditions for the invariants we found in Section \ref{Section invariants} to fully characterize the theory of the limit 2-semigroup, namely, the residue types of elements $\hat{\lang}$-definable without quantifiers, and the three ratios $r(a,b), r(b-\alpha(b),a)$, and $r(\alpha_1,ab)$. This result will follow from such theories admitting quantifier elimination, meaning the following:

\begin{dfn}
    A theory $T$ in a language $\mathcal L$ has \emph{quantifier elimination} if for every formula $\phi(\bar x)$ there is a quantifier free formula $\psi(\bar x)$ such that 
    \[
    T\models \forall \bar x\ [ \phi(\bar x)\Leftrightarrow \psi(\bar x)].
    \]
\end{dfn}

We will then show examples of choices of such invariants that completely characterize the theory of the limit 2-semigroup.

\medskip

We begin with some standard model theoretic results.

The following is Corollary 3.1.6 in \cite{Mar}. 

\begin{fact}
    Let $T$ be an $\mathcal L$-theory. Suppose that for all quantifier free formulas $\phi(\bar x, y)$ and all $\mathcal M, \mathcal N$ models of $T$, if $A$ is a common substructure of $\mathcal M$ and $\mathcal N$ and $\bar a\in A$ and $b\in \mathcal M$ are such that $\mathcal M\models \phi(\bar a, b)$, then there is some $c\in \mathcal N$ such that $\mathcal N\models \phi(\bar a, c)$.

    Then $T$ has quantifier elimination.
\end{fact}

The need for $A$ to be a substructure (this is, containing all the constants and closed under functions) is clear, since if it is not closed under functions it is quite easy to find counterexamples. Let $\langle A\rangle^{\mathcal M}$ be the smallest subset of $\mathcal M$ containing $A$ and the constants, that is closed under functions. 

\begin{dfn}\label{embedding-def}
    Let $\mathcal M$ and $\mathcal N$ be $\mathcal L$-structures. We will say that a map $f$ from $\mathcal M$ to $\mathcal N$ is a \emph{embedding} if for every quantifier free formula $\theta(\bar x)$ and any $\bar a\in \mathcal M$ we have 
    \[
    \mathcal M\models \theta(\bar a)\Leftrightarrow \mathcal N\models \theta(f(\bar a))
    \]
\end{dfn}

The above criterion is easily seen to be equivalent to the following.

\begin{fact}
    Let $T$ be an $\mathcal L$-theory. Suppose that for all quantifier free formulas $\phi(\bar x, y)$ and all $\mathcal M, \mathcal N$ models of $T$, if $\bar a$ is a \emph{finite tuple} in $\mathcal M$ and $f:\langle \bar a\rangle^{\mathcal M}\rightarrow \mathcal N$ is an embedding, and $b\in \mathcal M$ is such that $\mathcal M\models \phi(\bar a, b)$, then there is some $c\in \mathcal N$ such that $\mathcal N\models \phi(f(\bar a), c)$.

    Then $T$ has quantifier elimination.
\end{fact}

This can be made into a more algebraic notion using the concept of $\omega$-saturation.

\begin{dfn}
    Let $T$ be a theory. We will say that $\mathcal M\models T$ is \emph{$\omega$-saturated} if whenever we have a countable set of formulas $\{\phi_i(x,\bar a_i)\}_{i\in \omega}$ such that for every \emph{finite} choice $i_1, \dots , i_n$ we have 
\[
\mathcal M\models \exists x \bigwedge_{j=1}^n \phi_{i_j}(x,\bar a_{i_j}),
\]    
then there is some $b\in \mathcal M$ such that $\mathcal M\models \phi_{i}(b,\bar a_{i})$, for all $i$.
\end{dfn}

Intuitively, an $\omega$-saturated model is such that whenever realizing a countable set of formulas is finitely consistent (and therefore realized in the model), then there is an element of the model realizing them all. It is shown (in \cite{Mar} for example) that any theory has an $\omega$-saturated model. From this fact, using the definition of $\omega$-saturation and the concept of types, we get another criterion for quantifier elimination.

\begin{dfn}
    Let $\model$ be an $\lang$-structure, and let $A\subseteq \model$. Let $\lang_A$ be the language $\lang$ extended with constant symbols for each element $a\in A$. Denote by $Th_A(\model)$ the set of all $\lang_A$-formulas (without free variables) that are true in $\model$ (interpreting each constant symbol $a\in A$ naturally as the element $a$ in $\model$). A set $p$ of $\lang_A$-formulas in the free variables $x_1,\ldots,x_n$ is called an \emph{$n$-type} if $p\cup Th_A(\model)$ is consistent. For any tuple $\bar{m}\in \model$, we define its type in $\model$ over $A$, denoted $\tp^\model(\bar{m}/A)$ (simply written as $\tp^\model(m)$ if $A=\emptyset$), to be the set of all $\lang_A$-formulas $\phi(\bar{x})$ such that $\model\models \phi(\bar{m})$, and we define its \emph{quantifier free type}, denoted $\qftp^\model(\bar{m}/A)$, to be the subset of $\tp^\model(\bar{m}/A)$ that consists of the formulas with no quantifiers.
\end{dfn}

\begin{fact}\label{qefact}
    Let $T$ be an $\mathcal L$-theory.  Let $\mathcal M, \mathcal N$ be models of $T$ and $\bar a$ a \emph{finite tuple} in $\mathcal M$. Assume that $f:\langle \bar a\rangle^{\mathcal M}\rightarrow \mathcal N$ is an embedding, and that for every $b\in \mathcal M$ there is some $c\in \mathcal N$ such that the quantifier free type $\qftp^{\mathcal M}(\bar a, b)$ is equal to 
    $\qftp^{\mathcal N}(f(\bar a), c)$. 

    (This is, any quantifier free formula realized in $\mathcal M$ by $\bar a, b$ is realized in $\mathcal N$ by $f(\bar a), c$.)

    Then $T$ has quantifier elimination.
\end{fact}

\subsection{The $b$-property}\label{subsection b-property}

We will prove that the following condition (the $b$-property) is sufficient to have $T$ eliminate quantifiers. 
As a consequence, fixing invariants that result in an expansion of $T_{lons}$ satisfying the $b$-property will make such a theory have quantifier elimination, which will make the theory complete.

\begin{dfn}
    Let $T\supseteq T_{lons}$ be an $\hat{\lang}$-theory. We say $T$ has the $b$-property if for every $\model,\nodel\models T$, with $\nodel$ $\omega$-saturated, any $\hat{\lang}$-substructure $\model_0\leq\model$, any embedding $f\colon M_0\to N$, and any $x\in M_b(\model)$, there exists $y\in M_b(\nodel)$ such that:
    \begin{itemize}
        \item the $M_b$ and $R$-residue types of $x$ (in $\mathcal M$) and of $y$ (in $\mathcal N$) coincide,
        \item the $M_a$ and $R$-residue types of $\alpha(x)$ and of $\alpha(y)$ coincide,
        \item if $l$ is a definable distance in $M_0$ less than $a$, then 
        \[\mathcal M\models x-\alpha(x)<l \Leftrightarrow \mathcal N\models y-\alpha(y)<f(l),\] and
        \item for any $m\in M_0$ we have 
        \[\mathcal M\models x<m \Leftrightarrow \mathcal N\models y<f(m);\] in other words, \[\tp^\model(x/M_0)\big\vert_<=\tp^\nodel(y/f(M_0))\big\vert_<.\]
    \end{itemize} 

    We will say that $y$ \emph{witnesses the $b$-property} for $x$.
\end{dfn}

Some considerations. First, by saturation of $\mathcal N$ we just need to show that given $x$ and $M_0$ as above, and any finite set of formulas involving $M_b$-residues and residues of $x$, $M_a$-residues and residues of $\alpha(x)$, the order type of $x-\alpha(x)$ over definable distances less than $a$, and the order type of $x$, the corresponding formulas in $\mathcal N$ are realized by some element.

Now, any finite set of formulas in the order type of an element can of course be reduced to an interval. So:

\begin{obs}
    $T\supseteq T_{lons}$ has the $b$-property if and only if for any $\model,\nodel\models T$, any $\hat{\lang}$-substructure $\model_0\leq\model$, any embedding $f\colon M_0\to N$, if a given system of equations over $M_0$ is realized by some element $x\in \mathcal M$, then the corresponding system of equations (via $f$) must be realized by some $y\in \mathcal N$.
\end{obs}

Notice that this characterization doesn't need saturation for $\mathcal N$. Also, in light of Remark \ref{finite primes}, we get the following:

\begin{obs}\label{finite primes b-property}
    If the constant $a$ in $T\supseteq T_{lons}$ is such that no prime divides $a$ infinitely many times (recall that this means that for every prime $p$ there is an $n$ such that $T\models \neg R_{{p^n},0}(a)$ then $T$ has the $b$-property if and only if for any $\model,\nodel\models T$, with $\nodel$ $\omega$-saturated, any $\hat{\lang}$-substructure $\model_0\leq\model$, any embedding $f\colon M_0\to N$, if a given reduced system of equations over $M_0$ is realized by some element $x\in \mathcal M$, then the corresponding system of equations (via $f$) must be realized by some $y\in \mathcal N$.
\end{obs}

\subsection{Quantifier elimination}

Assuming it satisfies the $b$-property, we will prove quantifier elimination for the theory $T$ of 2-semigroups with infinite $a$ and $b$, and fixed invariants using the criterion in Fact \ref{qefact}. This is, given an embedding $f$ from a substructure $M_0$ to a saturated $\mathcal N$, and any $x\in \mathcal M$, we can extend $f$ to $\langle M_0, x\rangle$. This will be done in four stages. We will first prove the result whenever $x\in \mathcal M$ is such that 

\[M_b\left(\langle M_0, x\rangle\right)=M_b\left( M_0\right)\] and 
\[M_a\left(\langle M_0, x\rangle\right)=M_a\left( M_0\right).\]

We will then prove the extension of the embedding whenever $x\in M_a$ is such that 
\[M_b\left(\langle M_0, x\rangle\right)=M_b\left(\langle M_0\rangle\right).\]

These two cases will follow mainly from quantifier elimination of Presburger Arithmetic. 

We will then show we can extend the embedding for any $x\in M_b$. This is where the $b$-property hypothesis is used, and it is the heart of the proof. 

Lastly, we will use a standard argument using an enumeration of $\langle M_0, x\rangle$ to show that the above cases imply that one can extend the embedding $f$ for any $x\in \mathcal M$.

Regarding the extension of this embedding $f$, notice the following:
\begin{rem}\label{q_0 en M_a y M_b}
    If $r_0=r(a,b)\neq 0$, then there is no $x\in M_a(\model)\setminus M_0$ such that $M_b(\langle M_0,x\rangle)=M_b(M_0)$. If $x$ is such that $M_b(\langle M_0,x\rangle)=M_b(M_0)$, then we must have $\beta(x)\in M_0$, so that $\alpha(\beta(x))\in M_0$. Since $\beta(x)-b < x <\beta(x)$, we get that $\alpha(\beta(x)-b)=\alpha(\beta(x))-\alpha(b)+\epsilon(b) <  x \leq \alpha(\beta(x))$. Since $r_0\neq 0$, there is $n\in \N$ such that $\alpha(b)=n(a)$, which implies that $x\in M_0$.
\end{rem}

\subsubsection{The easy (and general) cases}

\begin{prop}\label{t-prop}
    Let $T\supseteq T_{lons}$ be an $\lang$-theory, $\model,\nodel\models T$, and $\phi\colon M_0\subseteq\model\to\nodel$ be an embedding, with $\model$ countable, and $\nodel$ saturated. Then, for every $x\in\model$ such that $M_a(\langle{x,M_0}\rangle)=M_a(M_0)$ and $M_b(\langle{x,M_0}\rangle)=M_b(M_0)$, there is $y\in\nodel$ such that $\tp^\model(x/M_0)=\tp^\nodel(y/\phi(M_0))$.
\end{prop}

\begin{proof}
Since $M_a(\langle{x,M_0}\rangle)=M_a(M_0)$ and $M_b(\langle{x,M_0}\rangle)=M_b(M_0)$, by unique decomposition 
\[
\{w\in \langle{x,M_0}\rangle\mid w\leq ab\}=
\{w\in \langle{M_0}\rangle\mid w\leq ab\}.  
\]
Since the domains of the functions $\alpha, \beta, \pi_a$ and $\pi_b$ are precisely the elements in $\model$ less than or equal to $ab$, we have that 
$\langle{x,M_0}\rangle=\langle{x,M_0}\rangle_{<,+,\frac{\cdot}{n}, -}$. This also implies that we don't have to worry about $M_a$-congruences nor $M_b$-congruences. In short, if we define 
$\mathcal L_{Pres}:=\{<,+,-\}\cup \{\frac{\cdot}{n}\}_{n\in \N}\cup \{R_{n,r}\}_{r,n\in \N, r<n}$ then it is enough to find $y\in \mathcal N$ such that 
\[
\qftp(x/M_0)\restriction_{\mathcal L_{Pres}}=\qftp(y/f(M_0))\restriction_{\mathcal L_{Pres}}
\]

But these are all predicates definable in Presburger Arithmetic. Since by unique decomposition we must have that $ab<x$ and elements past $ab$ in both $\mathcal M$ and $\mathcal N$ behave as rays of a model of Presburger arithmetic, saturation of $\mathcal N$ implies the existence of $y$. \end{proof}


When the ratio $q_0 = 0$, it is possible to extend a model of $T_{lons}$ to have new multiples of $a$ without having new multiples of $b$. For such cases, analogously to the $b$-property, we can find a suitable multiple of $a$ in a saturated model that allows a natural extension of an embedding between structures. Formally:

\begin{prop}\label{a-prop}
    Let $q_0 = 0$, and let $T\supseteq T_{lons}$ be an $\hat{\lang}$-theory, $\model,\nodel\models T$, and $f\colon M_0\subseteq\model\to\nodel$ be an embedding, with $\model$ countable, and $\nodel$ saturated. Then, for every $x\in M_a(\model)\setminus M_0$ such that $M_b(\langle{x,M_0}\rangle)=M_b(M_0)$, there is $y\in\nodel$ such that $\qftp^\model(x/M_0)=\qftp^\nodel(y/f(M_0))$  (in particular, for any definable distance $l$ in $M_0$ less than $b$, $M\models \beta(x)-x<l\Leftrightarrow \nodel \models \beta(y)-y<f(l)$).
\end{prop}

\begin{proof}
    Since $M_a(\nodel)$ is a model of bounded Presburger by Fact \ref{Presburger}, there exists $y\in \nodel$ with the same $\{<,M_a\}$-type over $M_a(f(M_0))$ as $x$ over $M_a(M_0)$. Now, notice that given an element of $M_0$ with unique decomposition $m(ab)+w_a+w_b$, we have that $x < m(ab)+w_a+w_b$ if and only if $x \leq m(ab)+w_a+\alpha(w_b)$, and since $m(ab)+w_a+\alpha(w_b)\in M_a^*(M_0)$, the order type of $x$ over $M_0$ is determined by its order type over $M_a(M_0)$, and similarly for $y$. Hence, the order type of $x$ over $M_0$ is that of $y$ over $f(M_0)$. By Theorem \ref{Axiom 12}, any definable distance in $M_0$ less than $b$ is of the form $\beta(m_a)-m_a$ for some $m_a\in M_a(M_0)$, so (assuming $\beta(m_a) < \beta(x)$, which implies $m_a < x$)
    $$\model \models \beta(x)-x < \beta(m_a)-m_a \Leftrightarrow \model \models \beta(x)-\beta(m_a) < x-m_a$$
    which happens if and only if
    $$\model \models \alpha(\beta(x)-\beta(m_a))+m_a\leq x$$
    and since $\beta(x)\in M_0$, such inequality is determined by the order type of $x$ over $M_a(M_0)$. The case $\beta(m_a) > \beta(x)$ is similar. The same reasoning applies to $y$, so the definable distances statement is proved. There are still formulas that need to be verified in order to conclude that $\qftp^\model(x/M_0)=\qftp^\nodel(y,f(M_0))$. Those follow a similar reasoning to Lemma \ref{inequalities} and Lemma \ref{localiso}.
\end{proof}

\subsubsection{Extending an embedding for $x\in M_b$}
As can be seen in Proposition \ref{t-prop}, any $x$ which does not introduce new elements into $M_a$ or $M_b$ is easier to handle, since it doesn't add anything to the domains of $\pi_a, \pi_b$ nor do we need to worry about the $M_a$- and $M_b$-residues. 

With this in mind, we will give a notation for the domain of elements with unique decomposition. This is the countable union of definable subsets of a limit 2-semigroup $\mathcal M$: 
\[
\mathbb U=\bigcup_{n\in \N} nab+M_a+M_b.
\]

We will also relativize $\mathbb U$ to substructures. Given a substructure $M_0$ of a limit 2-semigroups $\mathcal M$ we will define 
\[
\mathbb U(M_0)=\bigcup_{n\in \N} nab+M_a(M_0)+M_b(M_0).
\]

Notice that any element in $\mathbb U$ is in the domain of the functions $\pi_a^n$ and $\pi_b^n$ for some $n\in \N$.  In the proofs that follow, we will abuse notation and omit the superscript on the $\pi_a^n$ and $\pi_b^n$ functions. The proofs and lemmas presented with an unspecified superscript can be seen formally as schemes of proofs and lemmas done for all possible superscripts on such functions.

\medskip

The following two lemmas allow us to reduce the amount of formulas we need to inspect in order to understand the structure of a limit 2-semigroup by telling us precisely how $M_a(M_0)$ and $M_b(M_0)$ change when we add an element from $M_b$ to the substructure $M_0$.

\begin{lemma}\label{fractions}
    Let $\model\models T_{lons}, M_0\leq \model$, and let $x\in M_b(\model)$. We have that $M_b\langle x,M_0\rangle_{\hat{\lang}}=M_b\langle x,M_0\rangle_{\{+,-,\frac{\cdot}{n}\}}$, and $M_a\langle x,M_0\rangle_{\hat{\lang}}=M_a\langle\alpha(x),M_0\rangle_{\{+,-,\frac{\cdot}{n}\}}$.
\end{lemma}
\begin{proof}
  The proof will follow from the following claim.
    
    \begin{claim}
       Let $t\in \langle x,M_0\rangle_{\hat{\lang}}$. We have the following:
        \begin{itemize}
            \item  If $t\in\mathbb{U}\langle x,M_0\rangle_{\hat{\lang}}$ then 
            \[\pi_a(t),\alpha(t)\in M_a\langle \alpha(x),M_0\rangle_{\{+,-,\frac{\cdot}{n}\}},\] and  
            \[\pi_b(t),\beta(t)\in M_b\langle x,M_0\rangle_{\{+,-,\frac{\cdot}{n}\}}.\] 
        \item If $t\not\in\mathbb{U}\langle x,M_0\rangle_{\hat{\lang}}$ then $t=t_0+u_1-u_2$, where $t_0\in M_0\setminus \mathbb{U}(M_0)$, and $u_1,u_2\in \mathbb{U}\langle x,M_0\rangle_{\hat{\lang}}$ are such that $\pi_a(u_1),\pi_a(u_2),\alpha(u_1),\alpha(u_2)\in M_a\langle \alpha(x),M_0\rangle_{\{+,-,\frac{\cdot}{n}\}}$ and $ \pi_b(u_1),\pi_b(u_2),\beta(u_1),\beta(u_2)\in M_b\langle x,M_0\rangle_{\{+,-,\frac{\cdot}{n}\}}$.
        \end{itemize} 
    \end{claim}

\begin{claimproof}

The proof will be by induction on the length of the terms starting with all the elements of $M_0$ as constants. If $t=x$ or $t\in M_0$, the claim is evidently true.

\medskip
    
     Assume now that the claim holds for $t_1,t_2$. We will prove that all the elements generated by $t_1,t_2$ with a single use of a function in the language satisfy the claim. We will proceed by cases.  
     
     If $t_1,t_2\in\mathbb{U}\langle x,M_0\rangle_{\hat{\lang}}$, then we have $t_1=n(ab)+t_a+t_b, t_2=n'(ab)+t_a'+t_b'$, with $t_a,t_a',\alpha(t_b),\alpha(t_b')\in M_a\langle \alpha(x),M_0\rangle_{\{+,-,\frac{\cdot}{n}\}}$, and $t_b,t_b',\beta(t_a),\beta(t_a')\in M_b\langle x,M_0\rangle_{\{+,-,\frac{\cdot}{n}\}}$. By Remark \ref{alphafunction} we have that $\alpha(t_1+t_2)=\alpha(t_1)+\alpha(t_2)+\epsilon(a)$, where $\epsilon=0$ or $1$, and similarly for $\alpha(t_1-t_2)$, and the $\beta$ function. It's clear then that $t_1+t_2, t_1-t_2$ satisfy the claim, as well as $\alpha(t_1), \beta(t_1), \pi_a(t_1),\pi_b(t_1)$. Now, for the case $\frac{t_1}{n}$, note that $\alpha(t_1)-res_n^a(\alpha(t_1))(a) \leq t_1 < \alpha(t_1)+n(a)-res_n^a(\alpha(t_1))(a)$, so that $\frac{\alpha(t_1)-res_n^a(\alpha(t_1))(a)}{n}=\alpha(\frac{t_1}{n})$, where 
     for $\alpha\in M^a$ we define $res_n^a(k(ab)+\alpha):=res_n^a(k(ab))+res_n^a(\alpha)$. Analogously for the $\beta$ function. It follows that $\frac{t_1}{n}$ also satisfies the claim.

     If $t_1\not\in\mathbb{U}\langle x,M_0\rangle_{\hat{\lang}}$, but $t_2\in\mathbb{U}\langle x,M_0\rangle_{\hat{\lang}}$, the terms $t_1+t_2$ and $t_1-t_2$ clearly satisfy the claim. For the case $\frac{t_1}{n}$, by induction hypothesis let $t_0, u_1$ and $u_2$ be elements such that $t_1=t_0+u_1-u_2$ with the properties stated in the claim for this case. Let $0\leq i,j,i',j'<n$ be such that $\frac{n(ab)+n(u_1)+i(a)+j(b)}{n},\frac{n(ab)+n(u_1')+i'(a)+j'(b)}{n}$ both exist. We have that $\frac{t_1}{n}=\frac{t_0-(i-i')(a)+(j-j')(b)}{n} + \frac{n(ab)+n(u_1)+i(a)+j(b)}{n}-\frac{n(ab)+n(u_1')+i'(a)+j'(b)}{n}$ (the first fraction exists since the $R$-residue predicate is compatible with the sum and $R_{n,0}(t_1)$ holds), which satisfies the claim. 
     
     The last case is when $t_1,t_2\not\in\mathbb{U}\langle x,M_0\rangle_{\hat{\lang}}$. Again, by induction hypothesis, let $t_1=t_0+u_1-u_2 \geq t_2=t_0'+u_1'-u_2'$, be as stated on the claim. That $t_1+t_2$ satisfy the claim is immediate. For $t_1-t_2$, we have that $t_1-t_2=(t_0-t_0')+(u_1+u_1')-(u_2+u_2')$. If $t_0-t_0'$ exists, such element is in $M_0$ (since both $t_0,t_0'$ are in $M_0$), and so the claim will be satisfied. If $t_0-t_0'$ doesn't exist, we can write $t_1-t_2=(t_0+ab-t_0')+(u_1+u_1')-(ab+u_2+u_2')$ (where $t_0+ab-t_0'$ exists since their difference is greater than or equal to $ab$, and belongs to $M_0$), which clearly satisfies the claim. 
\end{claimproof}

Since by construction $M_a\langle\alpha(x),M_0\rangle_{\{+,-,\frac{\cdot}{n}\}}$ is closed under $+, -, \frac{\cdot}{n}$, that $M_a\langle x,M_0\rangle_{\hat{\lang}}=M_a\langle\alpha(x),M_0\rangle_{\{+,-,\frac{\cdot}{n}\}}$ follows immediately from the definition of $\hat{\lang}$ and the claim. Similarly for  $M_b\langle x,M_0\rangle_{\hat{\lang}}$.
\end{proof}

\begin{lemma}\label{inequalities}
    Let $\model,\nodel\models T_{lons}$ with $\nodel$ saturated, $M_0\leq \model$ be an $\hat{\lang}$-substructure, $f\colon M_0\to \nodel$ an embedding, $x\in M_b(\model)\setminus M_0$, and let $y\in \nodel$ be a corresponding witness of the $b$-property for $x$. Then, any inequality between linear combinations (over $\mathbb{Z}$) of $x$ and $\alpha(x)$ and elements of $M_0$ holds in $\model$ if and only if the inequality comparing $y$ and $\alpha(y)$ with the corresponding elements in $f(M_0)$ holds in $\nodel$.
\end{lemma}

\begin{proof}
    Since $\tp^\model(x/M_0)\big\vert_<=\tp^\nodel(y/f(M_0))\big\vert_<$, we also have $\tp^\model(\alpha(x),M_0)\big\vert_<=\tp^\nodel(\alpha(y),f(M_0))\big\vert_<$. 

The same argument as that of Remark \ref{q_0 en M_a y M_b} shows that 
$\model\models \alpha(x)<m$ for $m\in M_0$ if and only if $\model\models \alpha(x)+k(a)<m$ for all $k\in \N$, and that $\model\models x<m$ for $m\in M_0$ if and only if $\model\models x+k(b)<m$ for all $k\in \N$.

The analogous result holds of course for $\nodel$, since $f$ being an embedding implies that $y\not\in f(M_0)$. This implies that $\model \models \alpha(x)\leq m$ if and only if $\model \models \alpha(x)< m$.

   \medskip

Now, we need to show that 
\[
\model \models c_1(x)+c_2(\alpha(x))+t<c_1'(x)+c_2'(\alpha(x))+t'
\]
if and only if 
\[
\nodel \models c_1(y)+c_2(\alpha(y))+f(t)<c_1'(y)+c_2'(\alpha(y))+f(t')
\]
for $c_1, c_2, c_1', c_2'\in \mathbb Z$ and $t, t'\in M_0$. 

\medskip

We will first prove the result for $t,t'\in \mathbb U(M_0)$, so $t=n(ab)+t_a+t_b$ and $t'=n'(ab)+t_a'+t_b'$.

    Assume towards a contradiction that
    $$\model\models  c_1(x)+c_2(\alpha(x))+n(ab)+t_a+t_b<c_1'(x)+c_2'(\alpha(x))+n'(ab)+t_a'+t_b'$$ and $$\nodel\models c_1(y)+c_2(\alpha(y))+n(ab)+f(t_a)+f(t_b) > c_1'(y)+c_2'(\alpha(y))+n'(ab)+f(t_a')+f(t_b').$$
    for some $t_a,t_a'\in M_a(M_0),t_b,t_b'\in M_b(M_0), n, n',c_1,c_1',c_2,c_2'\in\N$. Applying the $\alpha$ function to both sides, and by Remark \ref{alphafunction}, we get
    $$\model\models c_1(\alpha(x))+c_2(\alpha(x))+n(ab)+t_a+\alpha(t_b)<c_1'(\alpha(x))+c_2'(\alpha(x))+n'(ab)+t_a'+\alpha(t_b')+\epsilon_x (a)$$ and
     $$\nodel\models c_1(\alpha(y))+c_2(\alpha(y))+n(ab)+f(t_a)+\alpha(f(t_b))> c_1'(\alpha(y))+c_2'(\alpha(y))+n'(ab)+f(t_a')+\alpha(f(t_b'))+\epsilon_y (a)$$ where $\epsilon_x$ and $\epsilon_y$ are elements in $\mathbb Z$. We now need to proceed by (sub)cases. 

\medskip

    If $c_1+c_2\neq c_1'+c_2'$, the terms involving $\alpha(x)$ can be grouped to one side without canceling out, and we can reduce the $\mathcal M$ equation to either $k(\alpha(x))<m_a$ or $k(\alpha(x))> m_a$, for some $m_a \in M_a^*(M_0)$ (the $\epsilon_x(a)$ term can be ignored by the initial observation); the $\mathcal N$ equation will reduce, respectively to $k(\alpha(y))>f(m_a)$ or $k(\alpha(y))< f(m_a)$. In either case, this contradicts the fact that the order type of $\alpha(x)$ over $M_0$ is that of $\alpha(y)$ over $f(M_0)$. 

\medskip
    
    Now, the other case is when all the $\alpha(x)$ terms cancel out (so that $c_1+c_2=c_1'+c_2'$), and we either have $c_1=c_1'$ and $c_2=c_2'$, or $c_1\neq c_1'$ and $c_2\neq c_2'$. In the first case, the result follows trivially by cancellation laws and the fact that $f$ is an embedding. In the second case, since $c_1-c_1'=c_2'-c_2$, if we set $c:=c_1-c_1'$ or $c:=c_1'-c_1$ (whichever is positive), and set all the $x$ terms to one side, the inequality 
    \[
    c_1(x)+c_2(\alpha(x))+n(ab)+t_a+t_b<c_1'(x)+c_2'(\alpha(x))+n'(ab)+t_a'+t_b'
    \]
    which we assumed is satisfied in $\model$ can be rewritten as 
    \[
    c(x)+n(ab)+t_a+t_b < c(\alpha(x))+n'(ab)+t_a'+t_b'    
    \]
    (or with the inequality sign flipped), which is equivalent to
    \[
    c(x-\alpha(x))<n'(ab)+t_a'+t_b'-(n(ab)+t_a+t_b)
    \]
    For simplicity, set $d:= n'(ab)+t_a'+t_b'-(n(ab)+t_a+t_b)$. Now, take $k\in\N$, with $k<c$, such that $d+k(\beta_1-\alpha_1)$ is divisible by $c$ (formally, the distance $d+k(\beta_1-\alpha_1)$ can be written as the difference of two elements in $M_0$, both having the same $R$-residue modulo $c$). We get that
    \[
    \model\models x-\alpha(x)<\frac{d+k(\beta_1-\alpha_1)}{c}
    \]
    Notice that we must also have that
    \[
    \model\models x-\alpha(x)<\frac{d+k(\beta_1-\alpha_1)-c(\beta_1-\alpha_1)}{c}
    \]
     since that is the next preceding definable distance divisible by $c$, which is realized as the difference of two elements in $M_0$ by Theorem \ref{Axiom 12}. Doing the same steps for the inequality in $\nodel$, we get that
    \[
    \nodel\models y-\alpha(y)>\frac{d-(c+k)(\beta_1-\alpha_1)}{c}
    \]
    which contradicts the fact that $y$ is a witness of the $b$-property for $x$. This concludes this case.

\medskip
    
    The last non-trivial case is when
    $$\model\models c_1(x)+c_2(\alpha(x))+t_1<c_1'(x)+c_2'(\alpha(x))+t_2$$
    where $t_1,t_2\in M_0\setminus\mathbb{U}(M_0)$ (if $t_2\in M_0\setminus\mathbb{U}(M_0)$ and $t_1\in \mathbb{U}(M_0)$, then $f(t_1)<f(t_2)$ follows by definition). 
    
    If $t_2-t_1$ or $t_1-t_2$ is an element of $\mathbb{U}(M_0)$, this case reduces to the previous case. If $t_2-t_1$ exists but is not in $\mathbb{U}(M_0)$, then the corresponding inequality holds in $\nodel$ since $\nodel\models f(t_2-t_1)>m(ab)$ for every $m\in\N$. The only other subcase is when $t_1-t_2$ nor $t_2-t_1$ exist. If $t_1>t_2$ but $t_1-t_2$ doesn't exist, we get
    $$\model\models c_1(x)+c_2(\alpha(x))+(ab+t_1-t_2)<ab+c_1'(x)+c_2'(\alpha(x))$$
    (where $ab+t_1-t_2$ exists since their difference is greater than $ab$) and the case reduces to the initial case as well. The remaining cases follow the same reasoning, and the result is proved.
\end{proof}

We are now ready to prove that the $b$-property implies we can extend an embedding to a structure including a new element $x\in M_b$.

\begin{lemma} \label{localiso}
    Let $\model,\nodel\models T_{lons}$ with $\nodel$ saturated, $M_0\leq \model$ be an $\hat{\lang}$-substructure, $f\colon M_0\to \nodel$ an embedding, $x\in M_b(\model)\setminus M_0$ and let $y\in \nodel$ be a corresponding witness of the $b$-property for $x$. Then, the function $f'\colon=f\cup\{(x,y)\}$ extends naturally to an embedding $f'\colon \langle x, M_0\rangle_{\hat{\lang}}\to \nodel$.
\end{lemma}

\begin{proof}
    Using Lemma \ref{fractions}, it follows by induction that any element $t\in \mathbb{U}\langle x,M_0\rangle_{\hat{\lang}}$ can be written as $t=\frac{n(x)+m(\alpha(x))+u_0-u_1}{r}$, for some $n,m\in \Z$, and some $u_0,u_1 \in \mathbb{U}(M_0)$. Notice that if $t=\frac{n(x)+m(\alpha(x))+u_0-u_1}{r}=\frac{n'(x)+m'(\alpha(x))+u_0'-u_1'}{r'}$, we must have $r'n(x)+r'm(\alpha(x))+r'(u_0)-r'(u_1)=rn'(x)+rm'(\alpha(x))+r(u_0')-r(u_1')$, and by unique decomposition and the fact that $x\not\in M_b(M_0)$ and $\alpha(x)\not\in M_a(M_0)$, we get that $r'n=rn'$ and $r'm=rm'$, so that $\frac{n}{r}=\frac{n'}{r'}$, $\frac{m}{r}=\frac{m'}{r'}$, and $r'(u_0-u_1)=r(u_0'-u_1')$. So if we define the function $f'$ by $f'(\frac{n(x)+m(\alpha(x))+u_0-u_1}{r}):=\frac{n(y)+m(\alpha(y))+f(u_0)-f(u_1)}{r}$ there would be no ambiguity. Also, notice that since $f$ is an embedding, $x$ and $y$ have the same $b$ and $R$-residue type, and $\alpha(x)$ and $\alpha(y)$ have the same $a$ and $R$-residue type, the fraction $\frac{n(y)+m(\alpha(y))+f(u_0)-f(u_1)}{r}$ exists, so $f'$ is well-defined. 
    
    If $t\not\in\mathbb{U}\langle{x,M_0}\rangle_{\hat{\lang}}$, using Lemma \ref{fractions} again, an easy proof by induction shows that it can be written as $t=\frac{t_0+n(x)+m(\alpha(x))}{r}$ for some $t_0\in M_0\setminus \mathbb{U}(M_0)$, and some $n,m\in \Z$. Define $f'$ on such elements as $f'(\frac{t_0+n(x)+m(\alpha(x))}{r}):=\frac{f(t_0)+n(y)+m(\alpha(y))}{r}$. The proof that $f'$ is well defined in this case is argued almost in the same way as the previous case. 
    
\medskip

    We now check that $f'$ is indeed an $\hat{\lang}$-embedding. We need to check that $f'$ preserves all element relations and functions in $\hat{\lang}$ (constants are done since they are in $M_0$). 
    \begin{enumerate}[start=0]
    \item That $f'$ preserves +, -  and $\frac{\cdot}{n}$ follows from the definition.
      
        \item That $f'$ preserves all inequalities follows from Lemma \ref{inequalities} (and the axioms of semigroups).       
        
        \item That $f'$ preserves all residue predicates follows from the fact that $x$ and $y$ have the same $b$ and $R$-residue type, that $\alpha(x)$ and $\alpha(y)$ have the same $a$ and $R$-residue type, that residue predicates are compatible with the sum and difference, and the fact that $M_a\langle \alpha(x),M_0\rangle_{\{+,-,\frac{\cdot}{n}\}}\cong M_a\langle \alpha(y),f(M_0)\rangle_{\{+,-,\frac{\cdot}{n}\}}$, and $M_b\langle x,M_0\rangle_{\{+,-,\frac{\cdot}{n}\}}\cong M_b\langle y,f(M_0)\rangle_{\{+,-,\frac{\cdot}{n}\}}$.

        \item $f'(\alpha(\frac{n(x)+m(\alpha(x))+u_0-u_1}{r}))=\alpha(\frac{n(y)+m(\alpha(y))+f(u_0)-f(u_1)}{r}))$: Any element $x$ with unique decomposition $x=m(ab)+m_a+m_b$ is divisible by $n$ if and only if there is a non-negative integer $k\leq m$ such that $k(ab)+m_a$ and $(m-k)(ab)+m_b$ both are divisible by $n$. Since $n(x)+m(\alpha(x))+u_0-u_1$ is divisible by $r$, by grouping the $a$-component and $b$-component of $u_0$ and $u_1$, we can write $\frac{n(x)+m(\alpha(x))+u_0-u_1}{r} = \frac{k(ab)+m(\alpha(x))+ m_a}{r}+\frac{k'(ab)+n(x)+ m_b}{r}$, for some $m_a\in M_a(M_0),m_b\in M_b(M_0)$, and some integers $k,k'$, such that each of the fractions exist. Now, $\alpha(\frac{k(ab)+m(\alpha(x))+m_a}{r}+\frac{k'(ab)+n(x)+ m_b}{r})=\frac{k(ab)+m(\alpha(x))+ m_a}{r}+\alpha(\frac{k'(ab)+n(x)+ m_b}{r})$. We first show that $f'(\alpha(t(x)))=\alpha(t(y))$ for any $t\in\mathbb{N}$. We have that $\alpha(t(x))=t(\alpha(x))+c(a)$ if and only if $(c+1)(a)>t(x)-t(\alpha(x))>c(a)$, if and only if $\frac{(c+1)(a)}{t}>x-\alpha(x)>\frac{c(a)}{t}$. Since $y$ is a witness of the $b$-property for $x$, we must have $\frac{(c+1)(a)}{t}>y-\alpha(y)>\frac{c(a)}{t}$, so $\alpha(t(y))=t(\alpha(t))+c(a)$. Hence, $f'(\alpha(t(x)))=\alpha(t(y))$. Now, notice that for any $m_b$ multiple of $b$ divisible by $r$, we have that $\alpha(\frac{m_b}{r})=\frac{\alpha(m_b)+i(a)}{r}$, where $0\leq i < r$ is such that $R^a_{r,r-i}(\alpha(m_b))$ holds, since $\frac{\alpha(m_b)+i(a)}{r}-\frac{m_b}{r}<\frac{a+i(a)}{r}\leq a$. Therefore, $\alpha(\frac{k'(ab)+n(x)+ m_b}{r})=\frac{\alpha(k'(ab)+n(x)+m_b)+i(a)}{r}$, where $0\leq i < r$ is such that $R_{r,n-i}^a(\alpha(k'(ab)+n(x)+m_b))$ holds. Now, given any $m_b\in M_b^*(M_0)$, and any integer $t$, we have that $\alpha(t(x)\pm m_b)=\alpha(t(x))\pm \alpha(m_b) +c(a)$ (for some integer $c$) if and only if $(c+1)(a)>t(x)\pm m_b - (\alpha(t(x))\pm \alpha(m_b))>c(a)$ if and only if $(c+1)(a)>t(x-\alpha(x))+c'(a)\mp \alpha(m_b) \pm m_b>c(a)$ (where $\alpha(t(x))=t(\alpha(x))+c'(a)$), if and only if $(c+1-c')(a)\pm \alpha(m_b)\mp m_b>t(x-\alpha(x))>(c-c')(a)\pm \alpha(m_b)\mp m_b$ if and only if $\frac{(c+1-c')(a)\pm \alpha(m_b)\mp m_b}{t}>x-\alpha(x)>\frac{(c-c')(a)\pm \alpha(m_b)\mp m_b}{t}$ (the case $t<0$ is done analogously). Since $x-\alpha(x)$ and $y-\alpha(y)$ have the same order type over every realizable distance less than $a$ definable in $M_0$, $y-\alpha(y)$ also satisfies such inequality, and we get that $f'(\alpha(t(x)\pm m_b))=\alpha(t(y)\pm f(m_b))$. Lastly, we have that $f'(\alpha(\frac{k'(ab)+n(x)+m_b}{r}))=f'(\frac{\alpha(k'(ab)+n(x)+m_b)+i(a)}{r})=\frac{f'(\alpha(k'(ab)+n(x)+m_b)+i(a))}{r}=\frac{\alpha(k'(ab)+n(y)+f(m_b))+i(a)}{r}=\alpha(f'(\frac{k'(ab)+n(x)+m_b}{r}))$.
        \item $f'(\beta(\frac{n(x)+m(\alpha(x))+u_0-u_1}{r})=\beta(\frac{n(y)+m(\alpha(y))+f(u_0)-f(u_1)}{r})$: The proof is quite similar to the proof of 3, so we will not include it.
        \item Compatibility with the projection functions $\pi_a,\pi_b$ is straightforward to check from the definition of $\pi_a,\pi_b$ and $f'$, and the fact that $f$ is an $\hat{\lang}$-embedding.
    \end{enumerate}
\end{proof}

\begin{rem}\label{localiso-others}
    Using the ideas of the previous lemma, one can also extend the embedding $f:M_0\to N$ to include a new element $x\not\in\mathbb{U}(\model)$ simply by mapping $x$ to some $y\in \nodel$ with the same $R$-residue type as $x$, and same order type (over $f(M_0)$), which exists by Proposition \ref{t-prop}. The embedding can also be extended to include a new element $x\in M_a(M_0)$ by mapping $x$ to an element $y$ whose quantifier-free type over $f(M_0)$ is that of $x$ over $M_0$, which exists by Proposition \ref{a-prop}, and also follows a very similar proof. 
\end{rem}

\subsubsection{Proof of quantifier elimination}

\begin{theo}\label{qe}
    Let $T\supseteq T_{lons}$ be a $\lang_{ons}$-theory with the $b$-property. Then, $T$ has quantifier elimination.
\end{theo}
\begin{proof}
    We use the characterization of a theory having quantifier elimination as stated in Fact \ref{qefact}. Let $\model,\nodel\models T$ with $\model$ countable, $\nodel$ saturated, let $f_0:M_0\subseteq \model\to \nodel$ be an embedding, and let $m\in\model$. 
    
    Let $\{x_n\}_{n\in\omega}$ be a sequence of elements in $M_b(\langle M_0, m\rangle)$ such that 
    \[M_b\left(\left\langle M_0,x_1, x_2\dots\right\rangle_{\hat{\lang}}\right)=M_b(\langle M_0, m\rangle)_{\hat{\lang}}.\]
     We extend $f_0$ iteratively so that its domain is $M_1:=\langle M_0,x_1, x_2\dots\rangle_{\hat{\lang}}$ as follows: Suppose $f_{k}:\langle M_0,x_1\dots,x_k\rangle_{\hat{\lang}}\to\nodel$ is an embedding. Let $y_{k+1}\in\nodel$ be a corresponding witness of the $b$-property for $x_{k+1}$. By Lemma \ref{localiso}, if we define $f_k(x_{k+1})=y_{k+1}$, $f_k$ extends naturally to an $\hat{\lang}$-embedding $f_{k+1}:\langle M_0,x_1\dots,x_k,x_{k+1}\rangle\to\nodel$. We continue iteratively this process, and take $\hat{f_1}=\cup_{i=1}^\infty f_i$. Clearly, such $\hat{f_1}:\langle M_0,x_1,\dots\rangle_{\hat{\lang}}\to\nodel$ will be an embedding. 

\medskip

    Next, let $\{y_n\}_{n\in\omega}$ be a sequence of elements in $M_a\left(\left\langle M_1, m\right\rangle\right)\setminus \langle M_1, m\rangle$ such that 
    \[M_a\left(\left\langle M_1,y_1, y_2\dots\right\rangle_{\hat{\lang}}\right)=M_a(\langle M_1,m\rangle)_{\hat{\lang}}.\] 

    Since $ M_1\subset \left\langle M_0, m\right\rangle$ and  \[M_b\left(\left\langle M_0,x_1, x_2\dots\right\rangle_{\hat{\lang}}\right)=M_b(\langle M_0, m\rangle)_{\hat{\lang}},\] we have that for every $i$
    \[M_b\left(\left\langle M_1,y_1, \dots, y_i\right\rangle_{\hat{\lang}}\right)=M_b\left(\left\langle M_1,y_1, \dots, y_i, y_{i+1}\right\rangle_{\hat{\lang}}\right),\] Remark \ref{localiso-others} implies that $\hat{f_1}$ can be increased at each stage, so we get an embedding $\hat{f_2}$ from $M_2:=\left\langle M_1,y_1, \dots, y_i\right\rangle_{\hat{\lang}}$ to $\mathcal N$.
    
    If $m\not\in M_2$, we must have 
    \[\mathbb{U}\left(\left\langle M_2,m\right\rangle\right)=\mathbb{U}\left(M_2\right).\]  Again by Remark \ref{localiso-others}, there is an $\hat{\lang}$-embedding $\hat{f}$ extending $\hat{f_2}$ to a domain including $m$. By Fact \ref{qefact}, $T$ has quantifier elimination. 
\end{proof}

\section{Some complete theories of limit 2-semigroups}\label{Section: conclusions}
As mentioned at the end of Section \ref{Section: model theoretic}, completeness (or even consistency) of the theories given \emph{any} combination of invariants is beyond what we can prove in this paper. But we will prove one special case.

\begin{dfn}
 Let $\langle a_i, b_i\rangle_{i\in \N}$ 
be a sequence of relatively prime natural numbers with $a_i<b_i$. We will say that \emph{the residue of $b_i$ modulo $a_i$ is constant and equal to $n/m$} with $m,n\in \N$ if and only if for every pair $a_i, b_i$ we have $mb_i\equiv_{a_i} n$. We will require that $(m,a_i)=1$ for cofinitely many $a_i$, and that $\lim_{i\rightarrow \infty} a_i=\infty$.     
\end{dfn}

In this section we will prove the following:

\begin{theorem}\label{rational b}
    Let $\langle a_i, b_i\rangle_{i\in \N}$ be a sequence of relatively prime natural numbers with $a_i<b_i$ and such that the residue of $b_i$ modulo $a_i$ is constant and equal to $m/n$. 
    
    Assume that for every prime $p$ there is some $n$ such that the set $\{i \mid p^n|a_i\}$ is finite (we will refer to this condition as ``no prime divides $a$ infinitely many times'').
    
    Let $\U$ be a non principal ultrafilter over $\N$ and $S_i$ the numerical semigroup generated by $a_i, b_i$. Then the theory of $\prod S_i/\U$ is completely determined by:
    \begin{itemize}
        \item The axioms in Definition \ref{dfn axioms}.
        \item The residues realized by $a$.
        \item $m(b)=\beta_n$.
        \item The quotient $q_0$. 
        \item If $q_0=0$, the list of $M_a$-residues of $\alpha(b)$.
    \end{itemize} 
\end{theorem}

Notice that if $a_i, b_i$ all satisfy the conditions in the statement of Theorem \ref{rational b}, and $b:=\prod b_i/\U$ and $a:=\prod a_i/\U$, then $\prod S_i/\U$ satisfies the axioms in Definition \ref{dfn axioms}, and $m(b)=\beta_n$. So we only need to show that the set of consequences of the list of axioms above is a complete theory.

We will first proof that all the invariants specified in Section \ref{Section invariants} are determined. 

\begin{prop}\label{invariants for rational b}
    Let $\model$ be a limit $2$-semigroup such that $m(b)=\beta_n$ for some $m,n\in \N$ with $(\res_m(a),m)=1$. Then all the invariants in Section \ref{Section invariants} are determined by 
    the residues realized by $a$, the quotient $q_0$ and, if $q_0=0$, the $a$-residues of $\alpha(b)$.

    Furthermore, the only possible values of $q_0$ are $0$ or $m/k$ for some $k\in \N$.
\end{prop}

\begin{proof}
Suppose that $m(b)=\beta_n$. 

\medskip

$\alpha_n=\beta_n-n=\alpha(\beta_n)$ and we have $\alpha(m(b))=m(\alpha(b))+t(a)$ for some $t<m$.

So $n=\beta_n-\alpha(\beta_n)=m(b)-m(\alpha(b))-t(a)$ and
\[
n+t(a)=m(b-\alpha(b)).
\]
This implies that $t$ is the unique number less than $m$ such that $t(\res_m(a))\equiv_m -n$, which is completely determined by $\res_m(a)$. Also, the equality $m(b)-m(\alpha(b))=n+t(a)$ implies that $q_1=r(b-\alpha(b),a)=t/m$,  determines all the ratios $r(l(b)-\alpha(l(b)),a)$ and therefore determines the order types of  $\langle a, b, ab\rangle_{+, -, \frac{\cdot}{n}}$.

\medskip

On the other hand, $\beta_n-\alpha_n=n$, so by Theorem \ref{Axiom 12} if $\beta'=\beta_n+l(ab)$ and $\alpha'=\alpha_n+l(ab)$ with $l$ the unique number such that $\res^b_n(\beta_n+l(ab))=0$, then $n(\beta')-n(\alpha')=\beta_n-\alpha_n=n$. It follows that $n(\beta_1)=\beta_n+l(ab)$ and $n(\alpha_1)=\alpha_n+l(ab)$ with $l$ depending only on $\res_n(a)$ and $\res_n(b)$. So $n(\beta_1)=m(b)+l(ab)$ which implies that $q_2=r(\beta_1,ab)=l/n$ and, more generally, determines the complete order type of $\beta_1$ and $\alpha_1$ over $\langle a, ab\rangle_{+, - \frac{\cdot}{n}}$ and $\langle b, ab\rangle_{+, - \frac{\cdot}{n}}$, respectively.

\medskip

Finally, there are two cases. If $q_0\neq 0$ then $k(a)<m(b)<(k+1)(a)$ so $\alpha_n=k(a)$ which implies  $r(a,b)=m/k$ and $m(b)-k(a)=n$. So in this case, the residue type of $b$ is completely determined by the residue type of $a$. If $q_0=0$ then we do need to know the $M_a$-residue type of $\alpha(b)$. But once we know this, we know the residue type of $\alpha(b)$, and the residue type of $b$ (since $n+t(a)=m(b-\alpha(b))$), of $\beta_n=m(b)$, of $\alpha_n=\beta_n-n$ and then, since the residue type of $a$ and $b$ imply that of $ab$, we also have the residue type of $\beta_1$, since $n\beta_1=\beta_n+l(ab)$. 

\bigskip

This completes the proof of the proposition.
\end{proof}

We will now prove Theorem \ref{rational b} using Fact \ref{Marker}; namely, that any theory with prime models and quantifier elimination is complete. 

By Proposition \ref{invariants for rational b} we have all the invariants and by Subsection \ref{Subsection Prime Models}, $M^0:=\bigcup_{n\in \mathbb N} n(ab)+M^0_a+M^0_b$ where $M_a^0=\langle a, \alpha(b), ab\rangle_{+, - \frac{\cdot}{n}}$ and $M_b^0=\langle b, ab\rangle_{+, - \frac{\cdot}{n}}$ is a prime model. So, we are only left to prove that the theory of the ultraproduct of standard 2-semigroups $S_i$ with generators $a_i, b_i$, with rational residue of $b_i$ modulo $a_i$, and no prime dividing $a$ infinitely many times, has elimination of quantifiers.

\bigskip

So let $S$ be a limit $2$-semigroup such that 
\begin{itemize} 
    \item for any prime $p\in \N$ there is some $n$ such that $\neg R_{p^n, 0}(a)$ and 
    \item $\beta_n=m(b)$ for some $m,n\in \N$.
\end{itemize}

\begin{proof}(of Theorem \ref{rational b})    
   Let $a$ and $B$ be as in the statement of Theorem \ref{rational b}, and let $T$ be the theory implied by the axioms
    \begin{itemize}
        \item The axioms in Definition \ref{dfn axioms},
        \item The residues realized by $a$,
        \item $m(b)=\beta_n$,
        \item The quotient $q_0$, 
        \item If $q_0=0$, the list of $M_a$-residues of $\alpha(b)$.
    \end{itemize} 
    
    We need to show that $T$ is complete. By Theorem \ref{qe}, we just need to prove that any model of $T$ has the $b$-property. By Observation \ref{finite primes b-property} in Subsection \ref{subsection b-property} and our hypothesis, this amounts to showing that given models $\model, \nodel$ of $T$, a substructure $M_0$ of $\model$, and an embedding $f:M_0\rightarrow \nodel$, if a given reduced system of equations (as defined in Definition \ref{def: reduced system}) is realized by some element $x\in \model$, then the corresponding system of equations is realized in $\nodel$.

    Fix a reduced system of equations over some $M_0$. This is, 
    \begin{itemize}
   \item A (fixed) complete set of $M_b$-residues modulo $n$.
\item A complete set of residues modulo $n$ for $x-\alpha(x)$.
\item $a_1<x-\alpha(x)<a_2$.
\item $b_1< x<b_2$.
 \end{itemize}
 where $a_1, a_2$ are distances definable in $M_0$ and $b_1, b_2\in M_b(M_0)$.

 \medskip

 The model $\model$ is naturally a subset of some model of Presburger $\mathcal Z$. Let $\mathcal Z_a$ be the structure $([0,a), +_a, <)$ interpreted in $\mathcal Z$ with $a=a^{\model}$. 
 
 As in the discussion after Remark \ref{follows discussion}, if we identify $M_b(\model)$ with $\mathcal Z_a$, then the above system is realized in $\model$ whenever  there is some $x\in \mathcal Z_a$ satisfying:
 \begin{itemize}
 \item A (fixed) complete set of $M_b$-residues modulo $n$.
\item A complete set of residues modulo $n$ for $\lambda_b(x)$.
\item $a_1<\lambda_b(x)<a_2$.
\item $b_1< x<b_2$.
\end{itemize}
 where $\lambda_b(x)$ is (cyclic) multiplication in $\mathcal Z_a$ by the residue of $b=b^{\model}$ modulo $a=a^{\model}$.

 Now, $\beta_n=m(b)$ implies that the residue of $m(b)$ modulo $a$ is $n$, so that $\lambda_b$ is just multiplication by $n/m$ which is of course definable in  $\mathcal Z_a$. By Fact \ref{Garcia} the fact that this system is realized in $\mathcal Z_a$ is implied by the Presburger quantifier-free type of the tuple $a_1, a_2, b_1, b_2, a^\mathcal M$. Since $f$ is an embedding, the Presburger quantifier free type of the tuple $a_1, a_2, b_1, b_2, a^\mathcal M$ is the same as the Presburger quantifier free type of the tuple $f(a_1), f(a_2), f(b_1), f(b_2), a^{\nodel}$ which implies that any reduced system is realized in $\model$ if and only if the corresponding system is realized in $\nodel$, as required.
\end{proof}

\printbibliography

\end{document}